\documentclass{article}
\usepackage[utf8]{inputenc}
\usepackage{enumitem}

\usepackage{amsmath,amssymb,amsthm}
\usepackage{mathtools} 

\usepackage{subfigure}
\usepackage{dsfont}
\usepackage{color}
\usepackage{xcolor}
\usepackage{fullpage}
\usepackage[
natbib,
maxcitenames=3, 
maxbibnames=10,  
sorting=nyt,
sortcites,
defernumbers,
backend=biber,
backref,
doi=false,
url=false,
giveninits
]{biblatex}
\usepackage[labelfont=bf,textfont=md]{caption}
\usepackage{graphicx}
\usepackage{float}
\usepackage{todonotes}
\usepackage{bm}
\usepackage{tikz}
\usetikzlibrary{decorations.pathreplacing}
\usetikzlibrary{calc}
\usepackage{pgfplots}
\usetikzlibrary{pgfplots.fillbetween}

\usepackage{hyperref}
\hypersetup{colorlinks,
            breaklinks,
            linkcolor=blue,
            urlcolor=blue,
            anchorcolor=blue,
            citecolor=blue}
\usepackage[capitalize,nameinlink]{cleveref}
\crefname{assumption}{Assumption}{Assumptions}

\crefalias{step}{enumi}
\crefname{step}{Step}{Steps}

\crefalias{hyp}{enumi}
\crefname{hyp}{Hypothesis}{Hypotheses}

\usepackage{hhline}

\newcommand{\R}{\mathbb{R}}
\newcommand{\N}{\mathbb{N}}
\newcommand{\X}{\mathcal{X}}
\newcommand{\Y}{\mathcal{Y}}
\newcommand{\C}{\mathcal{C}}
\newcommand{\M}{\mathcal{M}}
\DeclareMathOperator{\E}{\mathbb{E}}
\newcommand{\eps}{\varepsilon}
\renewcommand{\epsilon}{\varepsilon}

\graphicspath{{./figures/}}

\DeclareMathOperator*{\esssup}{ess\,sup}
\DeclareMathOperator*{\essinf}{ess\,inf}
\DeclareMathOperator*{\argmin}{arg\,min}

\DeclareMathOperator{\dist}{dist}

\newcommand{\norm}[1]{\left\|#1\right\|}
\newcommand{\abs}[1]{\left|#1\right|}
\newcommand{\st}{\,:\,}
\newcommand{\de}{\,\mathrm{d}}

\DeclareMathOperator{\supp}{supp}

\newtheorem{theorem}{Theorem}[section]
\newtheorem*{theorem*}{Theorem}
\newtheorem{proposition}[theorem]{Proposition}
\newtheorem{lemma}[theorem]{Lemma}
\newtheorem{corollary}[theorem]{Corollary}
\newtheorem*{corollary*}{Corollary}
\theoremstyle{definition}

\newtheorem{example}{Example}
\theoremstyle{remark}
\newtheorem{remark}[theorem]{Remark}

\newcommand{\ksr}[1]{{\color{red}#1}}
\renewcommand{\ksr}[1]{#1}

\newcommand{\comment}[1]{}

\numberwithin{equation}{section}

\newcommand{\grad}{\nabla}

\renewcommand{\P}{\mathcal{P}}
\renewcommand{\L}{\mathcal{L}}
\renewcommand{\H}{\mathcal{H}}

\DeclareMathOperator{\Per}{Per}
\DeclareMathOperator{\TV}{TV}

\def\Xint#1{\mathchoice
{\XXint\displaystyle\textstyle{#1}}%
{\XXint\textstyle\scriptstyle{#1}}%
{\XXint\scriptstyle\scriptscriptstyle{#1}}%
{\XXint\scriptscriptstyle\scriptscriptstyle{#1}}%
\!\int}
\def\XXint#1#2#3{{\setbox0=\hbox{$#1{#2#3}{\int}$ }
\vcenter{\hbox{$#2#3$ }}\kern-.6\wd0}}
\def\intbar{\Xint-}

\title{Gamma-convergence of a nonlocal perimeter\\ arising in adversarial machine learning}
\author{
Leon Bungert
\thanks{Institute of Mathematics, Center of Artifical Intelligence and Data Science (CAIDAS), University of Würzburg, Emil-Fischer-Str. 40, 97074 Würzburg, Germany. \href{mailto:leon.bungert@uni-wuerzburg.de}{leon.bungert@uni-wuerzburg.de}} 
\and 
Kerrek Stinson
\thanks{Hausdorff Center for Mathematics, University of Bonn, Endenicher Allee 62, Villa Maria, 53115 Bonn, Germany. \href{mailto:kstinson@uni-bonn.de}{kstinson@uni-bonn.de}}
}
\date{\today}

\makeatletter
\let\blx@rerun@biber\relax
\makeatother

\begin{document}

\maketitle

\begin{abstract}
In this paper we prove Gamma-convergence of a nonlocal perimeter of Minkowski type to a local anisotropic perimeter. 
The nonlocal model describes the regularizing effect of adversarial training in binary classifications. 
The energy essentially depends on the interaction between two distributions modelling likelihoods for the associated classes. We overcome typical strict regularity assumptions for the distributions by only assuming that they have bounded $BV$ densities.
In the natural topology coming from compactness, we prove Gamma-convergence to a weighted perimeter with weight determined by an anisotropic function of the two densities. Despite being local, this sharp interface limit reflects classification stability with respect to adversarial perturbations.
We further apply our results to deduce Gamma-convergence of the associated total variations, to study the asymptotics of adversarial training, and to prove Gamma-convergence of graph discretizations for the nonlocal perimeter.
\end{abstract}

{
\hypersetup{linkcolor=black}
\tableofcontents
}
\section{Introduction}

While modern machine learning methods and in particular deep learning \cite{goodfellow2016deep} are known to be effective tools for difficult tasks like image classification, they are prone to adversarial attacks \cite{szegedy2013intriguing}.
The latter are imperceptible perturbations of the input which destroy classification accuracy.
As a way to mitigate the effects of adversarial attacks, Madry et al. in \cite{madry2017towards} suggested a robust optimization algorithm to train more stable classifiers. 
Given \ksr{a metric space $\X$ acting as feature space, a set $\Y$ acting as label space, a probability} measure $\mu\in\M(\X\times \Y)$, which models \ksr{the distribution of training data}, a loss function $\ell : \Y \times \Y \to \R$, and a collection of classifiers $\mathcal{C}$, adversarial training takes the form of the minimization problem
\begin{align}\label{eq:AT}
    \inf_{u \in \C} \E_{(x,y)\sim\mu}\left[\sup_{\tilde x \in B(x,\eps)}\ell(u(\tilde x), y)\right].
\end{align}
Here we use the notation $\E_{z\sim\mu}[f(z)]:=\int f(z)\de\mu(z)$.
Adversarial training seeks a classifier for which adversarial attacks in the ball of radius $\eps$ around $x$ (with respect to the metric on $\X$) have the least possible impact, as measured through the function $\ell$.
Here $\eps>0$ is referred to as the adversarial budget and will play an important role in this article.

Since its introduction a significant body of literature has evolved around adversarial training, both focusing on its empirical performance and improvement (see the survey \cite{bai2021recent}), and its theoretical understanding.
Since the purpose of this article is primarily of theoretical nature, we restrict our discussion to the latter developments.
As it turns out, adversarial training has intriguing connections within mathematics.
Firstly, it was noted in different works that adversarial training is strongly connected to optimal transport. 
This was explored in the binary classification case, where $\Y = \{0,1\}$ and $\ell$ equals the $0$-$1$-loss, in a series of works, see \cite{pydi2020adversarial,trillos2022adversarial,bhagoji2019lower,pydi2021many} and the references therein.
Only recently these results were generalized to the multi-class case where $\Y=\{1,2,\dots,K\}$ by Garc\'ia Trillos et al. in \cite{trillos2022multimarginal} who characterize the associated adversarial training problem in terms of a multi-marginal optimal transport problem.
Secondly, it was already observed by Finlay and Oberman in \cite{finlay2021scaleable} that asymptotically, meaning for very small values of $\eps>0$ in \labelcref{eq:AT}, adversarial training is related to a regularization problem, where the gradient of the loss function is penalized in a suitable norm:
\begin{align}\label{eq:reg_AT}
    \inf_{u \in \C} \E_{(x,y)\sim\mu}\left[\ell(u(x), y)\right] + \eps \E_{(x,y)\sim\mu}\left[\norm{\grad_x\ell(u(x), y)}_*\right].
\end{align}
Here $\X$ is assumed to be a Banach space and $\norm{\cdot}_*$ is the corresponding dual norm.
While these connections were mostly formal, C. and N. Garc\'ia Trillos in \cite{garcia2022regularized} made them rigorous in the context of adversarial training for residual neural networks. 
Still even there the relation between adversarial training and regularization was of asymptotic type, in particular, not allowing statements about the relation of minimizers of \labelcref{eq:AT,eq:reg_AT}. 
Furthermore, in \cite{trillos2022adversarial} Garc\'ia Trillos and Murray regard adversarial training in the form of \labelcref{eq:AT} for binary classifiers as evolution with artificial time $\eps$ and relate it to \labelcref{eq:reg_AT} with a perimeter regularization at time $\eps=0$.
A different approach was taken in their work together with the first author of this paper \cite{bungert2022geometry} where, again in the binary classification case \ksr{and for $\ell$ the $0$-$1$-loss}, it was shown that adversarial training is equivalent to a non-local regularization problem:
\begin{align}\label{eq:TV_AT}
    \labelcref{eq:AT}
    =
    \inf_{u \in \C} \E_{(x,y)\sim\mu}\left[\ell(u(x), y)\right] + \eps \TV_\eps(u;\bm{\rho}).
\end{align}
Here $\TV_\eps(\cdot;\bm{\rho})$ denotes a nonlocal total variation functional depending on \ksr{the measures $\bm{\rho}:=(\rho_0,\rho_1)$ defined as $\rho_i := \mu(\cdot\times\{i\}$) for $i\in\{0,1\}$ which are---up to normalization---the conditional distributions of the two classes describing their respective likelihoods.}
The set of classifiers can be the set of characteristic functions of Borel sets $\C_\mathrm{char} = \{\chi_A \st A \in \mathfrak B(\X)\}$ or the set of ``soft classifiers'' $\C_\mathrm{soft} = \{u : \X \to [0,1]\}$ which live in a Lebesgue space equipped with a suitable measure on $\X$.
In this paper existence of solutions for adversarial training was proven, which included suitable relaxations of the objective function in \labelcref{eq:TV_AT} to a lower semi-continuous function, the construction of precise representatives, and the insight that the model with $\C_\mathrm{soft}$ is a convex relaxation of the model with $\C_\mathrm{char}$.
Furthermore, regularity properties of the decision boundaries of solutions are investigated.

We would like to emphasize that the results in \cite{bungert2022geometry} are proved for \emph{open balls} $B(x,\eps)$ in \labelcref{eq:AT} and this will also be the setting of the present paper.
Usually, adversarial training is defined using closed balls which does not change the model drastically but requires more care with respect to measurability of the underlying functions, see the discussion in \cite[Remark 1.3, Appendix B.1]{bungert2022geometry}.
For a different approach to proving existence---using a closed ball model---we refer to the work \cite{awasthi2021existence} by Awasthi et al., and the follow-up paper \cite{awasthi2021calibration} studying consistency of adversarial risks.

The focus of this paper will be on the asymptotics of the functional $\TV_\eps(\cdot;\bm{\rho})$ in \labelcref{eq:TV_AT}.
In fact, we will work with the associated perimeter functional $\Per_\eps(A;\bm{\rho}) := \TV_\eps(\chi_A;\bm{\rho})$ \ksr{for $A\subset\X$} since all statements, in particular Gamma-convergence, proved for the perimeter directly carry over to the total variation (see \cref{sec:applications}).
This perimeter was shown in \cite{bungert2022geometry} to be of the form
\begin{equation}\label{eqn:perDefIntro}
\begin{aligned}
    \Per_\eps(A;\bm{\rho}) = 
    \frac{1}{\eps} 
    \int_\X 
    \left(\nu\text{-}\esssup_{B(x,\eps)}\chi_A - \chi_A(x)\right)\de\rho_0(x)
    +
    \frac{1}{\eps} 
    \int_\X
    \left(\chi_A(x)-\nu\text{-}\essinf_{B(x,\eps)}\chi_A\right)\de\rho_1(x),
\end{aligned}
\end{equation}
\ksr{where the essential supremum and infimum are taken with respect to a suitable measure $\nu$ with sufficiently large support.}
When $\X=\R^d$ and $\rho_i$ equal the Lebesgue measure, the perimeter \labelcref{eqn:perDefIntro} can be used to recover the Minkwoski content of subset of $\R^d$ by sending $\epsilon\to 0$, see \labelcref{eqn:minkowskiContent} below. In this simplified setting, the nonlocal perimeter has applications in image processing \cite{barchiesi2010variational} and is mathematically well-understood. 
A thorough study of its properties like isoperimetric inequalities or compactness was undertaken in \cite{cesaroni2017isoperimetric,cesaroni2018minimizers}, Gamma-convergence of related variants to local perimeters was investigated in \cite{chambolle2010continuous,chambolle2014remark}, and associated curvature flows were analyzed in \cite{chambolle2012nonlocal,chambolle2015nonlocal}. We note that Chambolle et al. \cite{chambolle2014remark} introduce anisotropy into the perimeter by replacing the ball $B(x,r)$ in the definition of \labelcref{eqn:perDefIntro} by a scaled convex set $C(x,r)$. 

In this paper, we study the asymptotic behavior of the nonlocal perimeter \labelcref{eqn:perDefIntro} as $\eps\to 0$, using the framework of Gamma-convergence (see, e.g., \cite{braidesGamma2007,DalMasoBook}). This approach is widely used in applications to materials science (see, e.g., \cite{ambrosio-tortorelli-1990,ContiFonsecaLeoni-gammConv2grad,FJMrigid,ModicaMortola}) and is particularly applicable to the study of energy minimization problems depending on singular perturbations, where it can describe complex energy landscapes in terms of simpler understood effective energies. In the case of phase separation in binary alloys, energetic minimizers closely approximate minimal surfaces \cite{Modica87}. Likewise, we will relate local minimizers of the perimeter \labelcref{eqn:perDefIntro} to a weighted minimal surface, which has a transparent geometric interpretation. 

Though the nonlocal perimeter \labelcref{eqn:perDefIntro} is not a phase field approximation of the classical perimeter, similar analytic tools are helpful. 
The Ambrosio--Tortorelli functional was introduced as an elliptic regularization of the Mumford--Shah energy for image segmentation with the nature of approximation made precise via Gamma-convergence \cite{ambrosio-tortorelli-1990,braidesFreeDiscont}. 
From the technical perspective, our work is related to the results of Fonseca and Liu \cite{fonsecaLiu2017} where Gamma-convergence of a weighted Ambrosio--Tortorelli functional is proven. 
In their setting, they consider a density described by a bounded $SBV$ function with density uniformly bounded away from~$0$. 
In contrast to Chambolle et al. \cite{chambolle2014remark} and Fonseca and Liu \cite{fonsecaLiu2017}, a principal challenge in our setting will be understanding the interaction between the densities $\rho_0$ and $\rho_1$ in the energy \labelcref{eqn:perDefIntro} and how these give rise to preferred directions.

We assume that $\X= \Omega \subset \R^d$ and that the measures $\rho_i$ have densities with respect to the Lebesgue measure and are supported on some subset of the domain $\Omega$.
While for smooth densities Gamma-convergence is proven relatively easily (as in \cite{chambolle2014remark}), we only assume that the densities are bounded $BV$ functions. 
In this case, we prove that the Gamma-limit is an anisotropic and weighted perimeter of the form
\begin{align*}
    A \mapsto \int_{\partial A\cap\Omega} \beta\left(\nu_A;\bm{\rho}\right)\de\H^{d-1},
\end{align*}
where $\nu_A$ denotes the unit normal vector to the boundary of $A$. More rigorous definitions are given the \cref{sec:main_results}. 
Here, $BV$ functions are a natural space for the distributions as discontinuities are allowed, but they are sufficiently regular to be well-defined on surfaces and thereby prescribe interfacial weights.
We note that while the anisotropic dependence on the normal $\nu_A$ vanishes for continuous densities $\rho_i$, in the discontinuous case, the anisotropy provides a direct interpretation of the asymptotic regularization effect coming from adversarial training \labelcref{eq:AT} for small adversarial budgets (see \cref{ex:betaInterp,ex:bayes}). 

An interesting consequence of our Gamma-convergence result is the convergence of adversarial training \labelcref{eq:AT} as $\eps\to 0$ to a solution of the problem with $\eps=0$ with minimal perimeter.
Furthermore, the primary result in the continuum setting can be used to recover a Gamma-convergence result for graph discretizations of the nonlocal perimeter; a setting that is especially relevant in the context of graph-based machine learning \cite{GarcSlep15,GarcSlep16,garcia2022graph}. 
Our approach for this discrete to continuum convergence is in the spirit of these works and relies on $TL^p$ (transport $L^p$) spaces which were introduced for the purpose of proving Gamma-convergence of a graph total variation by Garc\'ia Trillos and Slep\v cev in \cite{GarcSlep15}, but have also been used to prove quantitative convergence statements for graph problems, see, e.g., the works \cite{calder2020rates,calder2022improved} by Calder et al.

The rest of the paper is structured as follows: 
In \cref{sec:main_results} we introduce our notation, state our main results, and list some important properties of the nonlocal perimeter.
\cref{sec:Gamma-convergence} is devoted to proving a compactness result as well as Gamma-convergence of the nonlocal perimeters. 
In \cref{sec:applications}, we finally apply our results to deduce Gamma-convergence of the corresponding total variations, prove conditional convergence statements for adversarial training, and prove Gamma-convergence of graph discretizations.

\section{Setup and main results}\label{sec:main_results}

\subsection{Notation}

The most important bits of our notation are collected in the following.

\paragraph{Balls and cubes}

For $x\in\R^d$ and $r>0$ we denote by $B(x,r):=\{y\in\R^d\st\abs{x-y}<r\}$ the open ball with radius $r$ around $x$.
Furthermore, $Q_{\nu}(x,r)$ is an open cube centered at $x\in\R^d$ with sides of length $r$ and two faces orthogonal to $\nu$. 
If $\nu$ is absent, the cube is assumed to be oriented along the axes.
We also define $Q_{\nu}^{\pm}(x,r) := \{y\in Q_{\nu}(x,r) : \pm \langle y-x,\nu \rangle > 0 \}$ and let $Q'(x,r)$ denote the $d-1$ dimensional cube centered at $x$ with sides of length $r$.

\paragraph{Measures and sets}

The $d$-dimensional Lebesgue measure in $\R^d$ is denoted by $\mathcal{L}^d$ and the $k$-dimensional Hausdorff measure in $\R^d$ as $\mathcal{H}^k$.
For a set $A\subset\R^d$ we denote its complement by $A^c := \R^d\setminus S$.
If $A\subset\Omega\subset\R^d$ is a subset of some fixed other set $\Omega$, we let $A^c$ denote its relative complement $\Omega\setminus A$.
The symmetric difference of two sets $A,B\subset\R^d$ is denoted by $A\triangle B := (A\setminus B)\cup (B\setminus A)$.
Furthermore, the characteristic function of a set $A\subset\R^d$ is denoted by 
\begin{align*}
    \chi_A(x) :=
    \begin{dcases}
    1 \quad&x \in A,\\
    0 \quad&x\notin A,
    \end{dcases}
\end{align*}
and by definition it holds $\mathcal{L}^d(A\triangle B) = \int_{\R^d}\abs{\chi_A-\chi_B}\de x$.
The distance function to a set $A\subset\R^d$ is defined as
\begin{align*}
    \dist(\cdot,A) : \R^d\to[0,\infty),
    \quad
    \dist(x,A) := \inf_{y\in A}\abs{x-y}.
\end{align*}

\paragraph{Measure-theoretic set quantities}

For $t\in[0,1]$ the points where a measurable set $A\subset\R^d$ has density $t$ are defined as
\begin{align*}
    A^t := \left\lbrace x\in \R^d \st \lim_{r\downarrow 0}\frac{\mathcal{L}^d(A\cap B(x,r))}{\mathcal{L}^d(B(x,r))}=t\right\rbrace.
\end{align*}
The Minkowski content $\mathcal{M}(A)$ of $A\subset\R^d$ is defined as the following limit (in case it exists):
\begin{equation}\label{eqn:minkowskiContent}
    \mathcal{M} (A) := \lim_{\epsilon\to 0}\mathcal{M}_\epsilon(A) : =\lim_{\epsilon\to 0}\frac{\mathcal{L}^d(\{x:\dist(x,A)<\epsilon\})}{2\epsilon}.
\end{equation}
We denote by $\partial^\ast A$ the reduced boundary of a set $A$. This is the set of points where the measure-theoretic normal exists on the boundary of $A$ \cite[Definition 3.54]{ambrosio2000functions}.

\paragraph{Functions of bounded variation}

For an open set $\Omega\subset\R^d$ we let $BV(\Omega)$ denote the space of functions of bounded variation \cite{ambrosio2000functions}.
Let $u\in BV(\Omega)$ and $M \subset \Omega$ be an $\H^{d-1}$-rectifiable set with normal $\nu$ (defined $\mathcal{H}^{d-1}$-a.e.). 
For $\H^{d-1}$-a.e. point $x\in M$, the measure-theoretic traces in the directions $\pm\nu$ exist and are denoted by $u^{\pm \nu}(x)$ \cite[Theorem 3.77]{ambrosio2000functions}.
These are the values approached by $u$ as the input tends to $x$ within the half-space $\{y\st \langle y-x,\pm \nu \rangle>0\}$, precisely,
\begin{align*}
    u^{\pm\nu}(x) := \lim_{r\to 0}\intbar_{B(x,r)\cap\{y\st\langle y-x,\pm \nu \rangle>0\}} u(y) \de y.
\end{align*}
We typically write $u^\nu$ instead of $u^{+\nu}$.
We denote by $u^+(x) := \max\{u^{\nu}(x),u^{-\nu}(x)\}$ the maximum of the trace \ksr{values}, and likewise $u^-(x):=\min\{u^{\nu}(x),u^{-\nu}(x)\}$ for the minimum. \ksr{Note that this notation is different from the one used in \cite[Definition 3.67]{ambrosio2000functions}.}
The standard total variation of a function $u\in L^1(\Omega)$ is denoted by $\TV(u)$ and satisfies $\TV(u)=|Du|(\Omega)$ for $Du$ the measure representing the distributional derivative ($+\infty$ if it is not a finite measure).
For $u\in BV(\Omega)$ we let $J_u$ denote its jump set, see, for example, \cite{ambrosio2000functions} for a definition.

\subsection{Main results}

Let $\Omega\subset\R^d$ be an open domain.
We consider two non-negative measures $\rho_0,\rho_1$ which are absolutely continuous with respect to the $d$-dimensional Lebesgue measure.
To simplify our notation we shall identify these measures with their densities from now on, meaning that $\de\rho_i(x) = \rho_i(x) \de x$ and $\rho_i\in L^1(\Omega)$ for $i\in\{0,1\}$.

We define the nonlocal perimeter of a measurable set $A\subset\Omega$ with respect to the measures $\bm{\rho}:=(\rho_0,\rho_1)$ and a parameter $\eps>0$ as
\begin{align}\label{eq:nonlocal_perimeter}
    \Per_{\eps}(A;\bm{\rho})
    :=
    \frac{1}{\eps} 
    \int_\Omega 
    \left(\esssup_{B(x,\eps)\cap\Omega}\chi_A - \chi_A(x)\right)\rho_0(x)\de x
    +
    \frac{1}{\eps} 
    \int_\Omega 
    \left(\chi_A(x)-\essinf_{B(x,\eps)\cap\Omega}\chi_A\right)\rho_1(x)\de x.
\end{align}
\ksr{It arises a special case of \labelcref{eqn:perDefIntro} by choosing $\X=\Omega$ and $\nu$ as the Lebesgue measure.}
Our main result that we prove in this paper is Gamma-convergence of the nonlocal perimeters \labelcref{eq:nonlocal_perimeter} to the localized version which preserves the apparent anisotropy of the energy. To motivate the correct topology for Gamma-convergence we first state a compactness property of the energies. 
\begin{theorem}\label{thm:gammaCompact}
Let $\Omega\subset\R^d$ be an open and bounded Lipschitz domain, and let $\rho_0,\rho_1\in BV(\Omega)\cap L^\infty(\Omega)$ satisfy $\essinf_\Omega(\rho_0+\rho_1)>0$. Then for any sequence $(\eps_k)_{k\in\N}$ with $\lim_{k\to\infty}\eps_k=0$ and collection of sets $A_k \subset \Omega$ with $\limsup_{k\to \infty}\Per_{\eps_k}(A_{k};\bm{\rho}) < \infty$, we have that up to a subsequence (not relabeled)
\begin{equation}\label{eqn:gammaCompact}
\chi_{A_k} \to \chi_A \text{ in }L^1(\Omega) \quad \text{ and } \quad  A\text{ is a set of finite perimeter.}
\end{equation}
\end{theorem}

\begin{remark}
The assumption that $\rho_0$ and $\rho_1$ belong to $BV(\Omega)$ is crucial for regularity of the set $A$ following from \cref{thm:gammaCompact}. To see this, fix a set $A\subset \Omega$ which is not a set of finite perimeter. Defining $\rho_0 = \chi_A$ and $\rho_1  = \chi_{A^c}$, we have that $\Per_\epsilon(A;\bm{\rho}) \equiv 0$, showing that \labelcref{eqn:gammaCompact} cannot hold. If one enforces the constraint $\essinf_\Omega \rho_i > 0$ on both densities, this problem is resolved, however this is an unreasonable constraint in the context of classification since it would require both classes to be entirely mixed. 
\end{remark}

Supposing now that $A$ is a set of finite perimeter (i.e., $\chi_A \in BV(\Omega)$) and $\rho_0,\rho_1 \in BV(\Omega)$, we define the function
\begin{align}\label{eq:beta_nu}
    \beta(\nu;\bm{\rho}) (x) :=
    \min\{\rho_0^{\nu} (x) + \rho_1^{\nu}(x) ,\,\rho_0^{-\nu}(x) + \rho_1^{-\nu}(x),\, \rho_0^{-\nu}(x) + \rho_1^{\nu}(x)\}\quad \text{ for } x \in \partial^* A,
\end{align}
where $\nu  = \frac{D\chi_A}{|D\chi_A|}$ is the measure-theoretic inner normal for $A$. We suppress dependence of $\beta(\nu;\bm{\rho})$ on $A$ as $\beta(\nu;\bm{\rho})$ is uniquely prescribed, in the sense that if $A_0$ and $A_1$ are two sets of finite perimeter then the definition \labelcref{eq:beta_nu} is $\mathcal{H}^{d-1}$-a.e. equivalent in $\left\{x\in \Omega: \frac{D\chi_{A_0}}{|D\chi_{A_0}|}(x) = \frac{D\chi_{A_1}}{|D\chi_{A_1}|}(x)\in S^{d-1}\right\}.$
Note that if $\rho_0$ and $\rho_1$ are also continuous, then it holds for all $\nu$ that
\begin{align}
    \beta(\nu;\bm{\rho}) = \rho_0 + \rho_1,
\end{align}
but for general $BV$-densities $\beta$ may be anisotropic.
\begin{example}\label{ex:betaInterp}
To understand the behavior of $\beta$, it is informative to consider two simple $1$-dimensional examples. Supposing that 
\begin{equation}\label{eqn:simpleBetaCase0}
\Omega = (-1,1), \ A = (-1,0), \  \rho_0 = \chi_{(0,1)},\text{ and } \rho_1 = \chi_{(-1,0)},
\end{equation} we can ask what the minimum energy of $A$ should be in the limit. Note that $\Per_\epsilon (A;\bm{\rho}) = 2$; this is the case for $\rho_0^{-\nu} + \rho_1^{\nu} =2$ in the definition of $\beta$. But if we shift $A$ to the right and define $A_\delta = (-1,\delta)$ with $\delta > 0$, for sufficiently small $\epsilon$, then $\Per_\epsilon (A_\delta;\bm{\rho}) = 1$, which amounts to picking up the minimum $\rho_0^{-\nu} + \rho_1^{-\nu} = 1$ in the definition of $\beta.$ Likewise, one could shift to the left to recover $\rho_0^{\nu} + \rho_1^{\nu} = 1$. However, we can never recover a minimum energy like $\rho_0^{\nu} + \rho_1^{-\nu} = 0$ as this would require orienting the boundary of set $A$ like the boundary of the set $(0,1)$, which it is not close to in an $L^1$ sense.
Consequently, the limit perimeter of $A$ is $\Per(A;\bm{\rho})=\min\{1,1,2\}=1$.
To see that the minimum of $\rho_0^{-\nu} + \rho_1^{\nu}$ is needed in the definition of $\beta$, one can simply swap the roles of $\rho_0$ and $\rho_1$ in the previous example. Precisely, for 
\begin{equation}\label{eqn:simpleBetaCase}
\Omega = (-1,1), \ A = (-1,0), \ \rho_0 = \chi_{(-1,0)},\text{ and } \rho_1 = \chi_{(0,1)},
\end{equation} we have that $\Per_\epsilon (A;\bm{\rho}) = 0 = \rho_0^{-\nu} + \rho_1^\nu$, and shifting $A$ increases the energy to $\rho_0^{\nu} + \rho_1^\nu = \rho_0^{-\nu} + \rho_1^{-\nu} = 1$.
Consequently, we have $\Per(A;\bm{\rho})=\min\{1,1,0\}=0$.
\end{example}

With the density $\beta$ defined, we may state the principal Gamma-convergence result of the paper.

\begin{theorem}\label{thm:gamma}
Let $\Omega\subset\R^d$ be an open and bounded Lipschitz domain, and let $\rho_0,\rho_1\in BV(\Omega)\cap L^\infty(\Omega)$ satisfy $\essinf_\Omega(\rho_0+\rho_1)>0$. Then
\begin{align*}
    \Per_\eps(\cdot;\bm{\rho})
    \overset{\Gamma}{\longrightarrow}
    \Per(\cdot;\bm{\rho})
\end{align*}
in the $L^1(\Omega)$ topology, where the weighted perimeter is defined by
\begin{align}\label{eq:local_perimeter}
    \Per(A;\bm{\rho}) := 
    \begin{dcases}
    \int_{\partial^\ast A\cap \Omega}
    \beta\left(\frac{D\chi_A}{\abs{D\chi_A}};\bm{\rho}\right)
    \de\H^{d-1}\quad&\text{if }\chi_A\in BV(\Omega),\\
    \infty\quad&\text{else}.
    \end{dcases}
\end{align}
\end{theorem}
We note that \cref{ex:betaInterp,thm:gamma} provide a direct interpretation for how the nonlocal perimeter \labelcref{eq:nonlocal_perimeter} selects a minimal surface for adversarial training \labelcref{eq:AT}. 
The fidelity term $\E_{(x,y)\sim\mu}\left[\ell(u(x), y)\right]$ in \labelcref{eq:AT} roughly wants to align $A$ with $\supp\rho_1$, for which a simple case is given by \labelcref{eqn:simpleBetaCase0}.
In this situation, the cost function in \labelcref{eq:AT} for $A=(-1,0)$ takes the value $2\eps$ and the perimeter picks up the value $\beta=2$. 
In contrast, for the sets $A_\delta=(-1,\pm\delta)$ for $\delta>\eps$ the cost function has the value $\delta+\eps$ and these sets recover the limiting perimeter with $\beta=1$.
Hence, adversarial regularization reduces the cost by effectively performing a preemptive stabilizing perturbation of the classification region.

\ksr{
\begin{remark}[Decomposition of the limit perimeter]
We remark that the limit perimeter may be decomposed into an isotropic perimeter plus an anisotropic energy living only on the intersection of the jump sets $J_{\rho_0}\cap J_{\rho_1}$ of the two densities. In particular, one can directly verify that
\begin{align}\label{eq:perimeter_decomp}
    \Per(A;\bm{\rho}) 
    = 
    \int_{\partial^*A \cap \Omega} \rho_0^- + \rho_1^- \de\H^{d-1} 
    + 
    \int_{\partial^* A \cap J_{\rho_0}\cap J_{\rho_1} \cap \Omega} 
    \min\left\lbrace
    \left(\rho_0^{-\nu_A}-\rho_0^{\nu_A}\right)_+, \left(\rho_1^{\nu_A}-\rho_1^{-\nu_A}\right)_+
    \right\rbrace
    \de\H^{d-1},
\end{align}
where $\nu_A := \frac{D\chi_A}{\abs{D\chi_A}}$ is the measure-theoretic inner unit normal of the boundary of $A$ and we use the notation $t_+:=\max(t,0)$ for $t\in\R$.
Note that the only case for which the non-isotropic term is different from zero is the case that $\rho_0^{\nu_A}<\rho_0^{-\nu_A}$ and $\rho_1^{\nu_A}>\rho_1^{-\nu_A}$.
In particular, if at least one of the densities is continuous, the anisotropy vanishes.
\end{remark}
}

The most important consequence of \cref{thm:gammaCompact,thm:gamma} is a convergence statement as $\eps\to 0$ for the adversarial training problem for binary classifier which takes the form
\begin{align}\label{eq:AT_Omega_main_results}
    \inf_{A\in\mathfrak B(\Omega)}\E_{(x,y)\sim\mu}\left[\sup_{\tilde x\in B(x,\eps)\cap\Omega}\abs{\chi_A(\tilde x)-y}\right].
\end{align}
For this we need some mild regularity for the so-called Bayes classifiers, i.e., the solutions of \labelcref{eq:AT_Omega_main_results} for $\eps=0$.
The condition is derived in \cref{ssec:AT} and stated in precise form in \labelcref{eq:source_condition} there.
Sufficient for it to hold is that the Bayes classifier with minimal value of the limiting perimeter $\Per(A;\bm{\rho})$ defined in \labelcref{eq:local_perimeter} has a sufficiently smooth boundary.
We have the following convergence statement for minimizers of adversarial training.
\begin{theorem}[Conditional convergence of adversarial training]\label{thm:convergence_AT}
Under the conditions of \cref{thm:gammaCompact,thm:gamma} and assuming the source condition \labelcref{eq:source_condition}, any sequence of solutions to \labelcref{eq:AT_Omega_main_results} possesses a subsequence converging to a minimizer of
\begin{align}\label{eq:perimeter_minimal_solution}
    \min\big\lbrace \Per(A;\bm{\rho}) \st A \in \argmin_{B\in\mathfrak B(\Omega)}\E_{(x,y)\sim\mu}[\abs{\chi_{B}(x)-y}]\big\rbrace.
\end{align}
\end{theorem}
\ksr{
\begin{example}\label{ex:bayes}
Another one-dimensional example highlights the effect of the anisotropy for the statement of \cref{thm:convergence_AT}.
For this let $\Omega:=(-2,2)$ and define the piecewise constant densities
\begin{align*}
    \rho_0(x) :=
    \begin{cases}
        \frac{3}{16}\quad&x\in(-2,-1),\\
        \frac{2}{16}\quad&x\in(-1,1),\\
        0\quad&x\in(1,2),
    \end{cases}
    \qquad
    \rho_1(x) :=
    \begin{cases}
        0\quad&x\in(-2,-1),\\
        \frac{2}{16}\quad&x\in(-1,1),\\
        \frac{5}{16}\quad&x\in(1,2).
    \end{cases}
\end{align*}
Note that all the sets $A\in\{(\alpha,2)\st\alpha\in[-1,1]\}$ minimize the Bayes risk $\mathbb E_{(x,y)\sim\mu}\left[\abs{\chi_A(x)-y}\right]$. 
However, easy computations show that $\Per((\alpha,2);\mathbf\rho)=\frac{4}{16}$ for $\alpha\in(-1,1]$ and $\Per((-1,2);\mathbf\rho)=\frac{3}{16}$.
Hence, $A=(-1,2)$ has the smallest perimeter among these minimizers and hence solves \labelcref{eq:perimeter_minimal_solution}.
In this case the optimal Bayes classifier saturates the entire support of $\rho_1$ at the cost of picking up the small jump $\lim_{x\uparrow -1}\rho_0(x)-\lim_{x\downarrow-1}\rho_0(x)=\frac{1}{16}$.

Note that if one symmetrizes the situation by setting the values of $\rho_0$ on $(-2,1)$ and of $\rho_1$ on $(1,2)$ to $\frac{4}{16}$, all the perimeters are the same. 
So in the previous situation it is the imbalance of the two classes which picks the support of $\rho_1$ as perimeter-minimal Bayes classifier and hence favors class $1$. 
\end{example}}

\cref{thm:gamma} is a direct consequence of \cref{thm:liminf2} for the $\liminf$ inequality and \cref{thm:limsup2} for the $\limsup$ inequality.
To prove \cref{thm:liminf2}, we proceed via a slicing method, which allows us to reduce the argument to an elementary, though technical, treatment of the $\liminf$ inequality in dimension $d=1$. To prove the $\limsup$ inequality in \cref{thm:limsup2}, we apply a density result of De Philippis et al. \cite{dePhilippis2017} to reduce to the case of a regular interface. In this setting, with the target energy in mind, we can locally perturb the interface to recover the appropriate minimum in the definition of $\beta(\nu;\bm{\rho})$ for the given orientation.

As is well known, \cref{thm:gamma} implies convergence of minimization problems involving $\Per_\epsilon(\cdot;\bm{\rho})$ to problems defined in terms of the local perimeter $\Per(\cdot ;\bm{\rho})$ (see, e.g., \cite{braidesGamma2007}).
Further, this result has a couple of important applications which we discuss in \cref{sec:applications}:
Firstly, it immediately implies Gamma-convergence of the corresponding total variation.
Secondly, as seen in \cref{thm:convergence_AT}, it has important implications for the asymptotic behavior of adversarial training \labelcref{eq:AT} as $\eps\to 0$.
Lastly, we can use it to prove Gamma-convergence of discrete perimeters on weighted graphs in the $TL^p$ topology, see \cref{thm:gamma_graph} in \cref{sec:applications}.

Future work will include using the results of this paper to study dynamical versions of adversarial training.
We remark that one way to think of the anisotropy in the limiting energy is that trace discontinuities pick up derivative information (in the spirit of $BV$). 
To recover an effective sharp interface model with anisotropy even for smooth densities, one could then look to the ``gradient flow" of the energy. 
For this one can interpret \labelcref{eq:TV_AT} as first step in a minimizing movement discretization of the gradient flow of the perimeter, where $\eps>0$ is interpreted as time step.
Iterating this and sending $\eps\to 0$ one expects to arrive at a weighted mean curvature flow, depending on the densities $\rho_i$.
Furthermore, preliminary calculations also indicate that the next order Gamma-expansion of the nonlocal perimeter, i.e., the expression $\frac1\eps\big(\Per_\eps(A;\bm{\rho}) - \Per(A;\bm{\rho})\big)$, relies on the gap between the trace values of $\rho_0$ and $\rho_1$ on $\partial A$ which induces anisotropy even for smooth densities.

\subsection{Auxiliary definitions and reformulations of the perimeter}

To work with the perimeter \labelcref{eq:nonlocal_perimeter}, it is convenient to reformulate the energy in a way such that it resembles the thickened sets introduced in the definition of the $d-1$ dimensional Minkowski content \labelcref{eqn:minkowskiContent}.
Recalling that $A^t$ denotes the points of density $t$ in $A$, for our setting, we have the following lemma, which may be directly verified (see also \cite{chambolle2014remark}).
\begin{lemma}\label{lem:properties_perimeter}
The perimeter \labelcref{eq:nonlocal_perimeter} of a measurable set $A\subset\Omega$ admits the following equivalent representation:
\begin{align}
    \Per_\eps(A;\bm{\rho})
    =
    \frac1\eps\rho_0(\lbrace x\in (A^1)^c\st\dist(x,A^1\cap \Omega)<\eps\rbrace)
    +
    \frac1\eps\rho_1(\lbrace x\in (A^0)^c\st\dist(x,A^0 \cap \Omega)<\eps\rbrace)
\end{align}
and furthermore it holds
\begin{align}
    \rho_1(\lbrace x\in (A^0)^c\st\dist(x,A^0\cap \Omega)<\eps\rbrace)
    =
    \rho_1(\lbrace x\in A^1\st\dist(x,(A^1)^c\cap \Omega)<\eps\rbrace),
\end{align}
meaning that $\Per_\eps(A;\bm{\rho})$ can be expressed in terms of $A^1$.
\end{lemma}
The above formulations provide a clear way to define restricted (localized) versions of the nonlocal perimeter.
For measurable subsets $A$ and $\Omega' \subset\Omega$, we define outer and inner nonlocal perimeters (respectively) of $A$ in $\Omega'$ as
\begin{subequations}\label{eq:outer_inner}
    \begin{align}
    \Per_\eps^0(A;\bm{\rho},\Omega') &:= \frac1\eps\rho_0(\lbrace x\in \Omega'\setminus A^1\st\dist(x,A^1\cap \Omega')<\eps\rbrace),
    \\
    \Per_\eps^1(A;\bm{\rho}, \Omega') &:=
    \frac1\eps\rho_1(\lbrace x\in \Omega'\setminus A^0\st\dist(x,A^0\cap \Omega')<\eps\rbrace),
    \\
    \Per_\eps^i(A;\bm{\rho}) &:= \Per_\eps^i(A;\bm{\rho},\Omega).
\end{align}
\end{subequations}

Note that by definition the monotonicity property $\Per_\eps^i(A;\bm{\rho}, \Omega_1)\leq\Per_\eps(A;\bm{\rho},\Omega_2)$ holds for subsets $\Omega_1\subset \Omega_2\subset\Omega$ and $i\in\{0,1\}$.
Furthermore, we define the restricted nonlocal perimeter of $A$ in $\Omega'$ as the sum
\begin{align}\label{eq:localized_nonlocal_perimeter}
    \Per_\eps(A;\bm{\rho}, \Omega') := \Per_\eps^0(A;\bm{\rho}, \Omega') + \Per_\eps^1(A;\bm{\rho}, \Omega').
\end{align}

\section{Gamma-convergence and compactness}\label{sec:Gamma-convergence}

In this section we will prove that the Gamma-limit of the nonlocal perimeters $\Per_\eps(A;\bm{\rho})$, defined in \labelcref{eq:nonlocal_perimeter}, is given by $\Per(A;\bm{\rho})$, defined in \labelcref{eq:local_perimeter}, thereby completing the proof of \cref{thm:gamma}.

\subsection{Compactness}

We can directly turn to the proof of compactness. The argument adapts the approach introduced in \cite[Theorem 3.1]{chambolle2014remark}, but takes care to account for the fact that the densities $\rho_i$ are allowed to vanish.

\begin{proof}[Proof of \cref{thm:gammaCompact}]
Let us define the following sequences of functions $(u_k)_{k\in\N}$ and $(v_k)_{k\in\N}$ by
\begin{equation}\label{def:uk}
\begin{aligned}
    u_k(x) := \left(1-\frac{\dist(x,A_k^1)}{\eps_k}\right) \vee 0,\qquad
    v_k(x) := \frac{\dist(x,(A_k^1)^c)}{\eps_k} \wedge 1,
    \qquad
    x\in\R^d,\;
    k\in\N.
\end{aligned}
\end{equation}
Here $u_k$ changes value in a small layer outside $A_k$, and similarly, $v_k$ transitions to $1$ inside $A_k$.
Recalling that the gradient of the distance function has norm $1$ almost everywhere outside of the $0$ level-set  (see, e.g., \cite{EvansGariepy}), up to a null-set, these functions satisfy
\begin{align*}
    \abs{\grad u_k} = \frac{1}{\eps_k}\chi_{\{x\in\R^d\st 0<\dist(x,A_k^1)<\eps_k\}},\qquad
    \abs{\grad v_k} = \frac{1}{\eps_k}\chi_{\{x\in\R^d\st 0<\dist(x,(A_k^1)^c)<\eps_k\}}.
\end{align*}
Since 
\begin{align*}
    \{x\in\ksr{\Omega}\st 0<\dist(x,A_k^1)<\eps_k\} &\subset \{x\in\Omega\setminus A^1_k\st \dist(x,A_k^1)<\eps_k\}, \\
    \{x\in\ksr{\Omega}\st 0<\dist(x,A_k^0)<\eps_k\} &\subset \{x\in A^1_k\st \dist(x,(A_k^1)^c)<\eps_k\}
\end{align*}
it follows
\begin{align}\label{ineq:grad_bounded_perimeter}
    \int_\Omega\abs{\grad u_k}\rho_0\de x \leq \Per_{\eps_k}^0(A_k;\bm{\rho})
    ,\qquad
    \int_\Omega\abs{\grad v_k}\rho_1\de x \leq \Per_{\eps_k}^1(A_k;\bm{\rho}).
\end{align}
By the hypothesis $\essinf_\Omega(\rho_0+\rho_1)>0$, we have $\rho_0 + \rho_1 >c_\rho$ almost everywhere in $\Omega$ for some constant $c_\rho>0$.
Hence, the sets $\Omega_i: =\{\rho_i > \delta\}$ cover $\Omega$ for $0<\delta<c_\rho/2$, i.e., $\Omega=\Omega_0\cup\Omega_1$. 
Applying the coarea formula to $\rho_0$ and $\rho_1$, we may choose $\delta$ such that $\Omega_i$ for $i=0,1$ are sets of finite perimeter.

Applying \ksr{the chain rule for $BV$ functions \cite[Theorem 3.96]{ambrosio2000functions} with $f:(y,z)\mapsto yz$}, we find
\begin{align*}
    \TV(u_k\chi_{\Omega_0}) 
    \leq 
    \int_{\Omega_0}\abs{\grad u_k}\de x
    +
    \TV(\chi_{\Omega_0}).
\end{align*}
Note that by construction $\TV(\chi_{\Omega_0})<\infty$.
Using this together with \labelcref{ineq:grad_bounded_perimeter} shows that $u_k \chi_{\Omega_0}$ is bounded uniformly in $BV(\Omega)$.
Similarly, we also have that $v_k \chi_{\Omega_1}$ is bounded uniformly in $BV(\Omega)$.
Consequently, we apply $BV$-compactness (see, e.g., \cite[Theorem 3.23]{ambrosio2000functions}) to both sequences, to find that $u_k \chi_{\Omega_0} \to u$ and $v_k \chi_{\Omega_1} \to v $ in $L^1(\Omega)$ where $u,v \in BV(\Omega)$.

Taking into account that
\begin{align*}
    \delta\int_{\Omega_0}
    \abs{u_k - \chi_{A_k}}\de x
    \leq 
    \int_{\Omega_0}
    \chi_{\ksr{\{x\in\Omega\setminus A^1_k\st \dist(x,A_k^1)<\eps_k\}}}\rho_0\de x
    \leq 
    \eps_k \Per_{\eps_k}^0(A_k)
    \to 0\quad \text{ as }k \to\infty,
\end{align*}
and a similar computation for $\norm{v_k-\chi_{A_k}}_{L^1(\Omega_1)}$, we further have $\chi_{A_k}\to u$ in $L^1(\Omega_0)$ and $\chi_{A_k}\to v$ in $L^1(\Omega_1)$.
Necessarily, it follows that $u = \chi_U $ and $v =  \chi_V $ where $U\subset \Omega_0$ and $V \subset \Omega_1$ are sets with $\chi_U = \chi_V$ in $\Omega_0\cap \Omega_1$. 
Using lower semi-continuity of the total variation we get
\begin{align*}
    \TV(\chi_U) \leq \liminf_{k\to\infty}\TV(u_k\chi_{\Omega_0}) < \infty
\end{align*}
and similarly $\TV(\chi_V)<\infty$, meaning that both sets have finite perimeter in $\Omega$.

Consequently, the set $A := U\cup V$ is of finite perimeter and satisfies $\chi_{A_k}\to \chi_A$ in $L^1(\Omega)$ since $\chi_{U} = \chi_V$ on $\Omega_0\cap\Omega_1$, and therefore
\begin{align*}
    \int_\Omega \abs{\chi_{A_k} - \chi_A}\de x
    &\leq 
    \int_{\Omega_0} \abs{\chi_{A_k} - \chi_U}\de x
    +
    \int_{\Omega_1} \abs{\chi_{A_k} - \chi_V}\de x
    \to 0\quad\text{ as } k\to\infty.
\end{align*}
\end{proof}

\subsection{Liminf bound}

We now prove the associated $\liminf$ bound in the definition of Gamma-convergence for \cref{thm:gamma}. The argument relies on slicing techniques, which allow the general $d$-dimensional case to be reduced to $1$-dimension. To keep this part of the argument relatively self-contained and notationally unencumbered, we perform the slicing argument locally, while remarking that the approach developed in \cite{braidesFreeDiscont} could also be applied. 

In our proof of the $1$-dimensional case, we will need the following auxiliary lemma, which is a direct consequence of Reshetnyak's lower semi-continuity theorem.

\begin{lemma}\label{lem:reshetnyakApplication}
\ksr{Let $(\eps_k)_{k\in\N}\subset(0,\infty)$ be a sequence of numbers with $\eps_k \to 0$.} Consider an interval $I\subset \R$, $\rho_0,\rho_1\in BV(I)$, and a sequence of sets $(A_k)_{k\in \N}$ with $A$ a set of finite perimeter in $I$ such that $\chi_{A_k} \to \chi_{A}$ in $L^1(I)$. It follows that 
\begin{equation}
\liminf_{\ksr{k\to \infty}} \Per_{\epsilon_k}^i (\ksr{A_k}{;\bm{\rho},I})\geq \int_{\partial^* A}\rho^-_i\, d\H^0.
\end{equation}
\end{lemma}
\begin{proof}
We restrict our attention to $i = 0$. Note, up to an equivalent representative, we may assume that $\rho_0 = \rho^-_0$ is a lower semi-continuous function defined everywhere (see \cite[Section 3.2]{ambrosio2000functions}). Recall the function $u_k$ introduced in the proof of \cref{thm:gammaCompact} in \labelcref{def:uk} for which we have
\begin{equation}\label{eqn:subEnergyRel}
\Per_{\ksr{\epsilon_k}}^0 (A_{\ksr{k}}{;\bm{\rho},I}) \geq \int_{I}|\nabla u_k |\rho_0\, \de x.
\end{equation} As $u_k \to \chi_A$ in $L^1(\{x:\rho_0>\delta\})$ for all $\delta > 0$, we may apply the Reshetnyak lower semi-continuity theorem \cite[Theorem 1.7]{spector2011simple} in each open set $\{x\in I:\rho_0>\delta\}$ to find
\begin{equation}\nonumber
\liminf_{k\to \infty}\int_{\{x\in I\st\rho_0>\delta\}}|\nabla u_k |\rho_0 \, \de x\geq \int_{\{x\in I\st\rho_0>\delta\}}\rho_0 \, d|D\chi| = \int_{\partial^* A\cap \{x\in I\st\rho_0>\delta\}}\rho^-_0\, d\H^0.
\end{equation}
Letting $\delta \to 0$ and applying \labelcref{eqn:subEnergyRel} concludes the result.
\end{proof}

\begin{theorem}\label{thm:liminf2}
Let $\Omega\subset \R^d$ be an open, bounded subset with Lipschitz boundary. Assume that $\rho_0,\rho_1$ belong to $ BV(\Omega) \cap L^\infty (\Omega)$.
Let $A\subset\R^d$ be a subset and $\{A_k\}_{k\in\N}\subset\R^d$ be a sequence of sets such that $\chi_{A_k}\to \chi_A$ in $L^1(\Omega)$ and $\chi_A \in BV_{loc}(\Omega)$.
Let $\{\eps_k\}_{k\in\N}\subset(0,\infty)$ be a sequence of numbers with $\eps_k \to 0$.
Then it holds that
\begin{align*}
   \Per(A;\bm{\rho})
   \leq \liminf_{k\to\infty}\Per_{\eps_k}(A_k;\bm{\rho}).
\end{align*}
\end{theorem}

\begin{proof}

For notational convenience, we suppress dependence on $\bm{\rho}$ and write $\Per_\eps(A)$ instead of $\Per_\eps(A;\bm{\rho})$. Likewise, we consider $\epsilon\to 0,$ with the knowledge that this refers to a specific subsequence.
Furthermore we define $\nu:=\frac{D\chi_A}{\abs{D\chi_A}}$ and write $\beta_\nu$ instead of $\beta\left(\frac{D\chi_A}{\abs{D\chi_A}};\bm{\rho}\right)$.

We split the proof into two steps. In Step 1, we show that the result holds in dimension $d=1.$ In this setting, a good representative of a $BV$ functions possesses one sided limits everywhere, which will effectively allow us to reduce to the consideration of densities given by Heaviside functions and an elementary analysis. To recover the $\liminf$ in general dimension, in Step 2, we use a slicing and covering argument along with fine properties of both $BV$ functions and sets of finite perimeter.  

\textbf{Step 1: Dimension $d=1$.} We may without loss of generality suppose that $\Omega$ is a single connected open interval and $\chi_A \in BV(\Omega ; \{0,1\})$. 
Consequently, the reduced boundary $\partial^* A$ of $A$ is a finite set. Supposing $x_0 \in \partial^* A $, we show for sufficiently small $\eta$ that 
\begin{equation}\label{eqn:LSC1d}
\liminf_{k\to\infty}\Per_{\eps_k}(A_k; x_0 + (-\eta,\eta))\geq  \beta_\nu(x_0).
\end{equation}
With this in hand, one can cover each element of $\partial^* A$ by pairwise-disjoint neighborhoods to apply the inequality \labelcref{eqn:LSC1d} to conclude the theorem in the case $d=1$.

We first suppose without loss of generality that $x_0 = 0$, $\nu(x_0) = -1,$ and that for any $\eta \leq \eta_0$ we have $\partial^* A \cap (-\eta, \eta) = \{0\}$. Up to choosing a smaller $\eta_0,$ we proceed by contradiction and suppose that there is a subsequence such that
\begin{equation}\label{eqn:contraLSC}
\lim_{\epsilon \to 0}\Per_{\eps}(A_\epsilon;(-\eta,\eta)) <\beta_\nu(0)
\end{equation}
for all $\eta \leq \eta_0.$
Applying \cref{lem:reshetnyakApplication}, it follows that 
\begin{equation}\label{eqn:resh1d}
\liminf_{\epsilon \to 0} \Per_\epsilon^i (A_\epsilon{;(-\eta,\eta)})\geq \min\{\rho_i^{-\nu} , \rho_i^{\nu}\}(0).
\end{equation}
This implies
\begin{align}\label{ineq:contradict_chain}
\begin{split}
\min\{\rho_0^{+\nu} + \rho_1^{+\nu} ,\,\rho_0^{-\nu} + \rho_1^{-\nu},\, \rho_0^{-\nu} + \rho_1^{\nu}\}(0)
&=
\beta_\nu(0)
\\
&>
\liminf_{\epsilon \to 0} \Per_\epsilon (A_\epsilon{;(-\eta,\eta)})
\\
&\geq 
\sum_{i=0}^1
\liminf_{\epsilon \to 0} \Per^i_\epsilon (A_\epsilon{;(-\eta,\eta)})
\\
&\geq 
\min\{\rho_0^{-\nu} , \rho_0^{\nu}\}(0)
+
\min\{\rho_1^{-\nu} , \rho_1^{\nu}\}(0).
\end{split}
\end{align}
Checking all the cases for the two minima on the right hand side shows that we must have
\begin{equation}\label{eqn:optimalRho}
    \rho_0^{\nu}(0) < \rho_0^{-\nu}(0)\quad \text{ and }\quad \rho_1^{-\nu}(0) < \rho_1^{\nu}(0) .
\end{equation}
and
\begin{align*}
    \rho_0^{\nu}(0) + \rho_1^{-\nu}(0) <  \beta_\nu(0).
\end{align*}
Using \labelcref{eqn:optimalRho}, we see that 
$$\rho_0^\nu(0) + \rho_1^\nu(0) < \rho_0^{-\nu}(0) + \rho_1^{\nu}(0) \quad \text{ and }\quad  \rho_0^{-\nu}(0) + \rho_1^{-\nu}(0) < \rho_0^{-\nu}(0) + \rho_1^{\nu}(0),$$
and so, without loss of generality, we may suppose that 
\begin{equation}\label{eqn:betaChoice}
\beta_\nu(0) = \rho_0^\nu(0) + \rho_1^\nu(0).
\end{equation}
Using \labelcref{eqn:betaChoice} inside \labelcref{eqn:contraLSC} and applying \labelcref{eqn:resh1d} with the minimum identified by \labelcref{eqn:optimalRho}, we have that 
\begin{align*}
    \rho_0^\nu(0) + \rho_1^\nu(0) 
    &=
    \beta_\nu(0)
    \\
    &>
    \liminf_{\epsilon \to 0} \Per_\epsilon^0(A_\epsilon;(-\eta,\eta)) + \limsup_{\epsilon \to 0}\Per_\epsilon^1(A_\epsilon;(-\eta ,\eta))
    \\
    &\geq 
    \rho_0^{\nu}(0) + \limsup_{\epsilon \to 0}\Per_\epsilon^1(A_\epsilon;(-\eta ,\eta)),    
\end{align*}
thereby showing that 
\begin{equation}\label{eqn:smallPer1}
\limsup_{\epsilon \to 0}\Per_\epsilon^1(A_\epsilon;(-\eta ,\eta)) < \rho_1^\nu(0) - \delta
\end{equation}
for some $\delta >0.$
As \labelcref{eqn:contraLSC} and \labelcref{eqn:smallPer1} are unaffected by choosing $\eta$ smaller, we now restrict $\eta$ to be $\eta<\delta /4$ and sufficiently small such that, for $i=0,1$, the one-sided limits are approximately satisfied, precisely,
\begin{equation}\label{eqn:oneSideCont}
\rho_i^{\pm \nu}(0)-\eta<\rho_i(x) \quad  \text{ for }\quad \pm x <0  \text{ and } x\in (-\eta,\eta).
\end{equation}
Using \cref{lem:properties_perimeter} it holds
\begin{equation}\label{eqn:perimRep}
\begin{aligned}
\Per^0_\epsilon(A_\epsilon) &= \frac{1}{\epsilon}\rho_{0}(\{x\in \Omega\setminus A^1_\epsilon \st \dist (x,A^1_\epsilon)<\epsilon\}), \\
\Per^1_\epsilon(A_\epsilon) &= \frac{1}{\epsilon}\rho_{1}(\{x\in A^1_\epsilon \st \dist (x,\Omega\setminus A^1_\epsilon)<\epsilon\}),
\end{aligned}
\end{equation}
with the key point being that $\Per_\eps^1$ can also be expressed in terms of the same underlying set $A^1_\epsilon$.
Using the above representation, by \labelcref{eqn:smallPer1} and \labelcref{eqn:oneSideCont}, for all $\epsilon$ sufficiently small, there is $x_\epsilon \in A_\epsilon^1$ such that $x_\epsilon \in (0,\eta).$
As $\chi_{A_\epsilon} \to \chi_A$ in $L^1(-\eta,\eta),$ we have that $|(A_\epsilon^1)^c\cap (0,\eta)| >\eta /2$ for all sufficiently small $\epsilon.$ Using $x_\epsilon$, the representation \labelcref{eqn:perimRep}, and \labelcref{eqn:oneSideCont},
it follows that $$ \Per^0_\epsilon(A_\epsilon) \geq \rho_0^{-\nu}(0) -\eta.$$
Using this in \labelcref{eqn:contraLSC} and applying \labelcref{eqn:resh1d} for $i=1$ with minimum determined via \labelcref{eqn:optimalRho}, we have
$$\beta_\nu(0) > \liminf_{\epsilon\to 0}\Per_\epsilon(A_\epsilon;(-\eta_0,\eta_0))\geq \rho_0^{-\nu}(0) -\eta + \rho_1^{-\nu}(0)$$ for all $\eta >0.$
Taking $\eta \to 0$, we have
$$\beta_\nu(0) >\liminf_{\epsilon\to 0}\Per_\epsilon(A_\epsilon;(-\eta_0,\eta_0))\geq \rho_0^{-\nu}(0) + \rho_1^{-\nu}(0),$$ contradicting the definition of $\beta_\nu$ in \labelcref{eq:beta_nu}.

\textbf{Step 2: Dimension $d> 1$.}
For any $\eta > 0$, we show that for $\H^{d-1}$-almost every $x_0\in \partial^* A$ there is $r_0 : =r_0(x_0,\eta)>0$ such that for every $r<r_0$ then 
\begin{equation}\label{eqn:cubeAlmostLiminf}
\int_{\partial^* A \cap Q_{\nu(x_0)}(x_0,r)} (\beta_\nu -\eta)\de\H^{d-1} \leq \liminf_{\epsilon\to 0} \Per_\epsilon(A_\epsilon;Q_{\nu(x_0)}(x_0,r)),
\end{equation}
where we recall that $Q_\nu(x_0,r)$ is a cube oriented along $\nu$.

Supposing, we have proven this, we may apply the Morse covering theorem \cite{fonseca2007modern} to find a countable collection of disjoint cubes $\{Q_{\nu(x_i)}(x_i,r_i)\}_{i\in \N}$ satisfying \labelcref{eqn:cubeAlmostLiminf} and covering $\partial^* A$ up an $\mathcal{H}^{d-1}$-null set. Directly estimating, we find that
$$ \liminf_{\epsilon\to 0} \Per_\epsilon(A_\epsilon) \geq \sum_{i<n}\liminf_{\epsilon\to 0} \Per_\epsilon(A_\epsilon;Q_{\nu(x_0)}(x_0,r)) \geq \int_{\partial^* A \cap \bigcup_{i<n}Q_{\nu(x_0)}(x_0,r)}(\beta_\nu -\eta)\de\H^{d-1}.$$
Taking $n \to \infty$ and then $\eta \to 0$ concludes the theorem.

Turning now to prove \labelcref{eqn:cubeAlmostLiminf}, we apply the De Giorgi structure theorem to conclude that up to a $\H^{d-1}$-null set, we have
$$\partial^* A = \bigcup_{i\in \N} K_i,$$ where $K_i$ is a subset of $C^1$ manifold. 
Consequently, for $\H^{d-1}$-almost everywhere $x_0\in \partial^* A$, we have that 
there is $i\in \N$ such that the density relations
\begin{equation}\label{eqn:densityRelprecur}
\lim_{r\to 0} \frac{\H^{d-1}(\partial^* A \cap Q_{\nu(x_0)}(x_0,r))}{r^{d-1}} = \ksr{\lim_{r\to 0}}\frac{\H^{d-1}( K_i \cap Q_{\nu(x_0)}(x_0,r))}{r^{d-1}} = 1  
\end{equation}
hold and the normals are aligned with $$\nu(x_0) = \nu_{K_i}(x_0).$$
Fixing $x_0$ as above, choose a scale $r_0$ such that the density relations in \labelcref{eqn:densityRelprecur} hold up to error $\eta,$ precisely,
\begin{equation}\label{eqn:densityRel}
\left|\frac{\H^{d-1}(\partial^* A \cap Q_{\nu(x_0)}(x_0,r))}{r^{d-1}}-1\right| + \left|\frac{\H^{d-1}(K_i \cap Q_{\nu(x_0)}(x_0,r))}{r^{d-1}}-1\right| \leq \eta \quad \text{ for all }r<r_0,
\end{equation}
and such that
\begin{equation}\label{eqn:nuCont}
  \|\nu_{K_i}(x) - \nu_{K_i}(x_0)\|\leq \eta\quad  \text{ for }x\in Q_{\nu(x_0)}(x_0,r_0)\cap K_i.
\end{equation}
After some algebraic manipulation, one sees that \labelcref{eqn:densityRel} implies
\begin{equation}\label{eqn:densityRel2}
\H^{d-1}(\partial^*A\setminus K_i \cap Q_\nu(x_0)(x_0,r))\leq 2\frac{\eta}{1-\eta} \H^{d-1}(\partial^*A\cap Q_\nu(x_0)(x_0,r)).
\end{equation}
Without loss of generality, we may assume $x_0  = 0$ and $\nu_{\partial^*A} (0) =e_d$.
 
 We perform a slicing argument. For notational convenience, we choose $A_\epsilon$ to be given by the equivalent representative $A^1_\epsilon$, allowing us to use the representation \labelcref{eqn:perimRep} without writing the superscript. 
 Defining 
\begin{equation}\label{eqn:Ay}
A^{y'} : = \{x_d\in \R \st (y',x_d) \in A\},
\end{equation}
note that for any set
\begin{equation}\label{eqn:subsetSliceRel}
\{y'\} \times \{x_d \in A^{y'} \st \dist(x_d , (A^{y'})^c)<\epsilon\} \subset \{x \in A \st \dist (x,A^c)<\epsilon\}\cap \{x' = y'\},
\end{equation}
where the distance is directly in $\R$ for the first set.
We now use the representation \labelcref{eqn:perimRep}, Fubini's theorem with $x = (x',x_d)$, and \labelcref{eqn:subsetSliceRel} to estimate
\begin{equation}\nonumber
\begin{aligned}
&\Per_\epsilon(A_\epsilon;Q(0,r)) 
=
\frac{1}{\eps}
\int_{Q(0,r) \cap \{x\in (A_\epsilon)^c\st \dist (x,A_\epsilon)<\epsilon\}}\rho_0 \de x 
+ 
\frac{1}{\eps}
\int_{Q(0,r)\cap \{x\in A_\epsilon\st \dist (x,(A_\epsilon)^c)<\epsilon\}}\rho_1 \de x \\
& \geq 
\int_{Q'(0,r)} \left(\frac{1}{\eps}\int_{(-r/2,r/2) \cap \{x_d\in (A_\epsilon^{x'})^c\st \dist (x,A_\epsilon^{x'})<\epsilon\}}\rho_0 \de x_d + \frac{1}{\eps}\int_{(-r/2,r/2)\cap \{x_d\in A_\epsilon\st \dist (x_d,(A_\epsilon^{x'})^c)<\epsilon\}}\rho_1 \de x_{d}\right) \de x'.
\end{aligned}
\end{equation}
\ksr{
Applying Fatou's lemma and Step 1 in the previous inequality, we have
\begin{equation}\label{eqn:almostFin}
\begin{aligned}
\liminf_{\epsilon\to 0} \Per_\epsilon  (A_\epsilon;Q(0,r)) & \geq \int_{Q'(0,r)}  \left(\int_{\partial^* A \cap \{x\st x_d \in (-r/2,r/2)\}} \beta_{\nu_{A^{x'}}} \de\H^{0}\right) \de x'.
\end{aligned}
\end{equation}
We note for almost every $x'\in Q'(0,r)$, if $x_d\in \partial^* A^{x'}$, then $x = (x',x_d)\in \partial^* A$ and $\langle (0,\nu_{A^{x'}}(x_d)), \nu(x)\rangle>0$ by \cite[Theorem 3.108 (a) and (b)]{ambrosio2000functions} applied to $\chi_A$. It then follows from \cite[Theorem 3.108 \ksr{(b)}]{ambrosio2000functions} applied to $\rho_0$ and $\rho_1$ that $ \beta_{\nu_{A^{x'}}}(x_d) = \beta_{\nu}(x)$ for $\de x' \otimes \mathcal{H}^0$-a.e. $x= (x',x_d).$
Using this, and subsequently the coarea formula \cite{ambrosio2000functions}, \labelcref{eqn:nuCont}, \labelcref{eqn:densityRel2}, and the $L^\infty$ bound on the densities ($\norm{\rho_i}_{L^\infty}\leq C$), we have
\begin{equation}\nonumber
\begin{aligned}
\int_{Q'(0,r)}  \left(\int_{\partial^* A \cap \{x\st x_d \in (-r/2,r/2)\}} \beta_{\nu_{A^{x'}}} \de\H^{0}\right) \de x' & \geq \int_{Q'(0,r)}  \left(\int_{K_i \cap \{x\st x_d \in (-r/2,r/2)\}} \beta_{\nu} \de\H^{0}\right) \de x' \\
& = \int_{Q(0,r) \cap K_i} \beta_\nu |\langle \nu_{K_i},e_d \rangle| \de\H^{d-1} \\
& \geq \int_{Q(0,r) \cap K_i} \beta_\nu (1-\eta) \de\H^{d-1} \\
& \geq  \int_{Q(0,r) \cap \partial^* A} (\beta_\nu - C\eta) \de\H^{d-1},
\end{aligned}
\end{equation}
which together with \labelcref{eqn:almostFin} concludes \labelcref{eqn:cubeAlmostLiminf} and the theorem.}\end{proof}

\subsection{Limsup bound}

In this section, we show that for a given target classification region $A\subset \Omega$, we can construct a recovery sequence with the optimal asymptotic energy. The result is precisely stated in the following theorem.

\begin{theorem}\label{thm:limsup2}
Let the hypotheses of \cref{thm:liminf2} hold. For any measurable set $A \subset \Omega$ and sequence $(\epsilon_k)_{k\in \N}$ with $\epsilon_k\to 0$ as $k\to \infty$, there is a sequence of sets $A_{k}$ such that $\chi_{A_k}\to \chi_A$ in $L^1(\Omega)$ and the following bound holds
$$\limsup_{k\to \infty}\Per_{\epsilon_k} (A_k ;\bm{\rho}) \leq \Per(A;\bm{\rho}).$$
\end{theorem}

The proof of this theorem relies on technical properties of BV functions, but is conceptually simple. 
Our approach is outlined in the following steps:
\begin{enumerate}
\item We use use a recent approximation result of De Philippis et al. \cite{dePhilippis2017} to approximate the $BV$ function $\chi_A$ by a function $u \in BV(\Omega)$ having higher regularity. We select a level-set, given by $A_\eta$, of $u$ such that $\partial A_\eta = \partial^*A_\eta$, $\mathcal{H}^{d-1}(\partial A_\eta \triangle \partial^* A) \ll 1$, and a large portion of $\partial A_\eta$ is locally given by a $C^1$ graph.

\item We then break $\partial A_\eta$ into a good set, with smooth boundary, and a small bad set, with controllable error. 

\item \label[step]{step:smoothRecov} In the good set, $A_\eta$ has $C^1$ boundary, and here we construct an almost optimal improvement by perturbing the smooth interface. This is the result of \cref{prop:smoothRecovery}.
\begin{enumerate}
\item Using a covering argument, we localize the construction of a near optimal sequence and introduce improved approximations in balls centered on points $x_0 \in \partial A_\eta$.

\item Depending on the minimum value of $\beta(\nu;\bm{\rho}) = \min\{\rho_0^{\nu} + \rho_1^{\nu},\rho_0^{-\nu} + \rho_1^{-\nu},\rho_0^{- \nu} + \rho_1^{ \nu} \}$, where $\nu$ is the inner normal of $A_\eta$, we either shift the interface up or down slightly or, in the latter case of the minimum, leave it unperturbed to recover the optimal trace energies. The essence of this is to approximately satisfy $\rho_0 + \rho_1 \approx \beta(\nu;\bm{\rho})$ on the modified interface.
\end{enumerate}

\item Diagonalizing on approximations for $A_\eta$ and then on $\eta$, one obtains a recovery sequence the original set $A$.
\end{enumerate}

We begin with the proof of \cref{prop:smoothRecovery} for \cref{step:smoothRecov}, and for this, a couple auxiliary lemmas will make the argument easier. In our construction, we will select an appropriate level-set using the following lemma, which says that given control on an integral one can control the integrand in a large region.
\begin{lemma}\label{lem:thetaFrac}
Let $f:(a,b) \to [0,\infty)$ be integrable and $\theta \in (0,1)$. Then
$$\theta (b-a) \leq \mathcal{L}^1 \left(\left\{y:f(y)\leq \frac{1}{1-\theta}\intbar_a^b f\de t \right\}\right) .$$
\end{lemma}
\begin{proof}
Using Markov's inequality one computes
\begin{align*}
    \mathcal{L}^1 \left(\left\{y:f(y)< \frac{1}{1-\theta}\intbar_a^b f\de t \right\}\right)
    &=
    (b-a)
    -
    \mathcal{L}^1 \left(\left\{y:f(y)\geq \frac{1}{1-\theta}\intbar_a^b f\de t \right\}\right) 
    \\
    &\geq 
    (b-a)-(1-\theta)(b-a) = \theta(b-a).
\end{align*}
\end{proof}
To control errors arising in our interface construction, we will take advantage of the assumption $\rho_i \in L^\infty (\Omega)$. Specifically, energetic contributions of $\mathcal{H}^{d-1}$-small pieces of our construction will be thrown away using the following proposition for the classical Minkwoski content.

\begin{proposition}[Theorem 2.106 \cite{ambrosio2000functions}]\label{prop:minkowskiContent}
Suppose that $f:\R^{d-1}\to \R^d$ is a Lipschitz map. Then for a compact set $K\subset\subset \R^{d-1}$ it holds 
$$\mathcal{M}(f(K)) = \mathcal{H}^{d-1}(f(K)),$$
where $\mathcal{M}$ is the $d-1$-dimensional Minkowski content defined in \labelcref{eqn:minkowskiContent}.
\end{proposition}

In the proof of \cref{step:smoothRecov} and in the case that the optimal energy sees the traces on both sides of the interface, we will need a recovery sequence with non-flat interface. We will apply the following lemma to show that the weighted perimeters converge on either side of the interface.

\begin{lemma}\label{lem:traceOnC1graph}
Suppose that $A\subset Q(0,r)$ is given by the sub-graph of \ksr{a function in $C^1(\overline{Q'(0,r)})$} which does not intersect the top and bottom of the cube $Q(0,r)$. Then
$$\lim_{\epsilon\to 0}\Per_\epsilon(A;\bm{\rho}, Q(0,r)) = \int_{\partial A \cap Q(0,r)}\left( \rho_0^{-\nu} + \rho_1^\nu \right) \de\H^{d-1}.$$
\end{lemma}
\begin{proof}
\ksr{Again, we suppress the dependency of $\bm{\rho}$.}
It suffices to show convergence for the outer perimeter $\Per_\epsilon^0 (A;Q(0,r))$ defined in \labelcref{eq:outer_inner}.

Let $g_0:Q'(0,r)\to (-r/2,r/2)$ be the function prescribing the graph associated with $\partial A$. Let $x \in \R^{d} \mapsto \operatorname{sdist}(x, A)$ denote the signed distance from $A$, with the convention that it is non-negative outside of~$A$. 
Note first that $\operatorname{sdist}^{-1}(s)\cap Q(0,\ksr{r-2s})$ is the graph of a Lipschitz function $g_s:Q'(0,\ksr{r-2s})\to (-r/2,r/2)$ with uniformly bounded gradient depending on $g_0$ for all $s>0$ sufficiently small. To see this, note that the level set $\operatorname{sdist}$ can be expressed by translating the graph of the boundary, that is, 
\begin{equation}\label{eqn:formgs}
g_s(x') = \sup_{(\nu',\nu_d) \in S^{d-1}}\{g_0(x'+s\nu')- s\nu_d\},
\end{equation} which is the supremum of equicontinuous bounded functions. 
\ksr{In fact, we have that 
\begin{equation}\label{eqn:convergenceClaim}
\|\nabla g_s - \nabla g_0\|_{L^\infty (Q'(0,r-2s))} \leq \omega(s),
\end{equation}
where $\omega$ is the modulus of continuity of $\nabla g_0$ in $Q(0,r)$. To see this, we note that for sufficiently small $t\in \R^{d-1}$ and $x' \in Q'(0,r-2s)$, by \labelcref{eqn:formgs}, there is always a $\nu_{t}\in S^{d-1}$ such that $g_s(x'+t) = g_0(x'+t+s\nu'_{t})- s\nu_{t,d}$. Using the mean value theorem, we can estimate
\begin{align*}
g_s(x'+t) - g_s(x') \leq & g_0(x'+t+s\nu'_{0})- s\nu_{0,d} -(g_0(x'+s\nu'_{0})- s\nu_{0,d}) \\
= & \langle\nabla g_0 (x' +\theta t + s\nu'_{0}), t\rangle \leq   \langle\nabla g_0 (x'), t\rangle + |t|\omega(|t| + s),
\end{align*}
where $\theta \in (0,1);$ the same bound holds from below. Now assuming that $x'$ is a point of differentiability for $g_s,$ we insert $g_s(x'+t) - g_s(x') = \langle\nabla g_s(x'),t\rangle + o(|t|)$ into the above inequality to find
$$|\langle\nabla g_s(x') - \nabla g_0 (x'),t\rangle|\leq o(|t|) + |t|\omega(|t|+s). $$
Fixing $r \in (0,1)$, taking the supremum over $t \in rS^{d-1}$, dividing by $r$, and then sending $r\to 0$, this inequality becomes
$$\|\nabla g_s(x') - \nabla g_0 (x')\|\leq \omega(s), $$ and as Lipschitz functions are differentiable almost everywhere, we recover \labelcref{eqn:convergenceClaim}.}

Define $\phi(x',s): = \phi_s(x') := (x',g_s(x'))$. We now apply the coarea and area formulas to rewrite the perimeter as 
\begin{equation}\nonumber
\begin{aligned}
\Per_\epsilon^0(A) = \frac{1}{\epsilon}\int_{\operatorname{sdist}^{-1}((0,\epsilon))}\rho_0 \de x = &\, \intbar_0^\epsilon \left[ \int_{\operatorname{sdist}^{-1}(s)}\rho_0(x) \de\H^{d-1}(x)\right]\de s \\
= & \,\intbar_0^\epsilon \left[ \int_{Q'(0,r)}\rho_0(\phi_s(x')) J_{x'}\phi_s(x')\de x'\right]\de s.
\end{aligned}
\end{equation}
Using the area formula once again, we have
$$\int_{\partial A \cap Q(0,r)} \rho_0^{-\nu}\de\H^{d-1} = \int_{Q'(0,r)}\rho_0^{-\nu}(\phi_0(x')) J_{x'}\phi_0(x')\de x'.$$
Taking the difference, we can estimate
\begin{equation}\nonumber
\begin{aligned}
&\left|  \Per_\epsilon^0(A) -\int_{\partial A \cap Q(0,r)} \rho_0^{-\nu} \de\H^{d-1} \right|  \\
&\leq \int_{Q'(0,r(1-\delta))}\left|\intbar_0^\epsilon \rho_0(\phi_s(x')) \de s  - \rho_0^{-\nu}(\phi_0(x'))\right| J_{x'}\phi_0(x')\de x'  \\
& \quad + C(\rho_0)\sup_{s\in (0,\epsilon)}\|J_{x'}\phi_s - J_{x'}\phi_0 \|_{L^\infty(Q'(0,r(1-\delta)))} r^{d-1} + C(\rho_0,\partial A)\left(r^{d-1}-((1-\delta) r)^{d-1}\right).
\end{aligned}
\end{equation}
Taking the $\limsup$ as $\epsilon\to 0$ and then letting $\delta\to 0 $ in the above estimate, we see that the lemma will be concluded if we show that
\begin{equation}\label{eqn:BVintTraceEst}
\limsup_{\epsilon\to 0}\int_{Q'(0,r(1-\delta))}\left|\intbar_0^\epsilon \rho_0(\phi_s(x')) \de s  - \rho_0^{-\nu}(\phi_0(x'))\right|\de x' = 0.
\end{equation}

Note that $(x',s)\mapsto \phi(x',s)$ is a bi-Lipschitz function on $Q'(0,r(1-\delta))\times (0,\epsilon)$ for sufficiently small $\epsilon.$
By \cite[Theorem 3.16]{ambrosio2000functions} on the composition of $BV$ functions with Lipschitz maps, it follows that $(x',s)\mapsto  \rho_0 \circ \phi(x',s)$ is a $BV$ function with 
\begin{equation}\label{eqn:gradCompEst}
|D(\rho_0 \circ \phi)|(Q'(0,r(1-\delta))\times (0,\epsilon)) \leq \|\nabla \phi^{-1}\|_{L^\infty}^{d-1}|D\rho_0|(\phi(Q'(0,r(1-\delta))\times (0,\epsilon))) \xrightarrow[\epsilon \to 0]{} 0.
\end{equation} 
Further, by fine properties of $BV$ functions (see, \cite[Theorem 3.108 \ksr{(b)}]{ambrosio2000functions}), it follows that 
\begin{equation}\label{eqn:covFineProp}
\left( \rho_0\circ \phi \right)^{e_d}(x',0) \ksr{= \lim_{s\downarrow 0} \rho_0(x',\phi_s(x')) = \lim_{s\downarrow 0} \rho_0(x',\phi_0(x')+s)} = \rho_0^{-\nu}(\phi_0(x')) 
\end{equation} for $x'\in \ksr{Q'(0,r)}$ almost everywhere.

Rewriting \cite[Eq. (3.88)]{ambrosio2000functions}, for $f \in BV(Q(0,r(1-\delta)))$, we have 
\begin{equation}\label{eqn:AFPtrace}
\int_{Q'(0,r(1-\delta))} \left|\intbar_0^\epsilon f(x',s) \de s - f^{e_d}(x,0)\right|\de x'\leq |Df|(Q'(0,r(1-\delta))\times (0,\epsilon)).
\end{equation}
Inserting $f = \rho_0\circ \phi$ into the above equation and using \labelcref{eqn:gradCompEst} and \labelcref{eqn:covFineProp} concludes the proof of \labelcref{eqn:BVintTraceEst} and thereby the lemma.
\end{proof}

We now prove that smooth sets have near optimal approximations, completing the proof of \cref{step:smoothRecov}.
As a matter of notation, we will typically consider the closure and boundary of a set $A$ relative to $\Omega$ and denote this as $\overline{A}$ and $\partial A$, respectively. However, to denote the closure of a set $A$ in $\R^d,$ we will write $\overline{A}^{\R^d}$. This distinction will be important to ensure that the energy does not charge the boundary of $\Omega.$

\begin{proposition}\label{prop:smoothRecovery}
Let $\Omega\subset \R^d$ be an open, bounded set with Lipschitz boundary, $M\subset \R^d$ a $C^1$-manifold without boundary, and $A \subset \Omega$ a set. Suppose that $\overline{\partial A}^{\R^d}$ is a submanifold of $M$, with the additional properties that $$\overline{\partial A}^{\R^d} = M \cap \overline{\Omega}\quad \text{ and } \quad \H^{d-1}(M\cap \partial \Omega) = 0.$$ Then for any $\eta>0,$ there is $A_\eta$ such that $A_\eta = A$ in a neighborhood of $\partial \Omega$ and the following inequalities hold:
\begin{equation}\nonumber
\begin{aligned}
\|\chi_{A} - \chi_{A_\eta}\|_{L^1(\Omega)}\leq & \ \eta\quad  \text{ and } \quad \limsup_{\epsilon \to 0} \Per_\epsilon (A_\eta;\bm{\rho},\Omega) \leq & \int_{\partial A} \ksr{\beta\left(\frac{D\chi_A}{\abs{D\chi_A}};\bm{\rho}\right)} \de\mathcal{H}^{d-1} + \eta.
\end{aligned}
\end{equation}
\end{proposition}

\begin{proof}
The primary challenge in this construction is controlling the interaction between the interfaces given by $\partial A$ and by $J_\rho : = J_{\rho_0}\cup J_{\rho_1}$, denoting the jump set of $\rho_0$ and $\rho_1$, with a secondary challenge being to ensure that $\partial A$ does not charge the boundary $\partial \Omega$. We denote the inner normal of $A$ by $\nu(x) : =\frac{D\chi_A}{|D\chi_A|}(x)$, \ksr{abbreviate $\beta_\nu=\beta\left(\tfrac{D\chi_A}{\abs{D\chi_A}};\bm{\rho}\right)$, and suppress the dependency on $\bm{\rho}$} for notational simplicity.

We write $\partial A$ as the union of a good surface and a bad surface $$\partial A  = S_G \cup S_B : = [\partial A \setminus J_\rho] \cup [\partial A \cap J_\rho].$$ We will select neighborhoods of Lebesgue points useful to our construction, and as such, we are only interested in keeping track of properties of $S_G$ and $S_B$ up to an  $\H^{d-1}$-null set.
We may assume the following are satisfied up to $\H^{d-1}$-null sets:
\begin{enumerate}[label={\arabic*.)}]
\item \label[hyp]{hyp:graph} For each point $x\in \partial A$, there is $r_x > 0$ such that up to a rotation and translation, $Q_{\nu(x)}(x,r_x)\cap A$ is given by the subgraph of a $C^1$ function $g_x$ centered at the origin with $\nabla g_x(0) = 0$.

\item \label[hyp]{hyp:massonSB}  For each point $x \in \partial A$,
$$  \lim_{r \to 0}\frac{\H^{d-1}( \partial A \cap Q_{\nu(x)}(x,r))}{r^{d-1}}= 1, $$
where we recall $Q_{\nu}(x,r)$ is a cube oriented in the direction $\nu$.

\item \label[hyp]{hyp:LebBeta}  For $x \in \partial A$,
$$ \lim_{r \to 0}\intbar_{\partial A \cap Q_{\nu(x)}(x,r)} |\beta_\nu(y) - \beta_\nu(x)| \de\H^{d-1}(y) = 0.$$

\item \label[hyp]{hyp:LebBulk} For $x\in S_G$, we have $$\lim_{r\to 0} \intbar_{Q_{\nu(x)}(x,r)} |\rho_i(y) - \rho_i(x)|\de y = 0.$$

\item   For $x \in S_G$,
$$ \lim_{r \to 0}\intbar_{\partial A \cap Q_{\nu(x)}(x,r)} |\rho^{\pm\nu}_i(y) - \rho_i(x)| \de\H^{d-1}(y) = 0. $$
 
\item For $x\in S_B$, $\nu(x) = \pm \nu_{J_\rho}(x) $, where $\nu_{J_\rho}$ denotes a unit normal of the $d-1$-rectifiable set $J_{\rho}.$
 
\item \label[hyp]{hyp:SBLeb} For $x\in S_B$, we may assume
\begin{equation}\nonumber
\begin{aligned}
\lim_{r\to 0} \intbar_{Q_{\nu(x)}^{\pm}(x,r)} |\rho_i(y) - \rho_i^{\pm\nu}(x)|\de y = 0,
\end{aligned}
\end{equation} 
where $Q_{\nu}^{\pm}(x,r) := \{y\in Q_{\nu}(x,r) : \pm \langle y-x,\nu \rangle > 0 \}.$

\item  \label[hyp]{hyp:goodOrientation} For $x \in S_B$,
$$ \lim_{r \to 0}\intbar_{\partial A \cap Q_{\nu(x)}(x,r)} |\rho^{\pm\nu}_i(y) - \rho_i^{\pm\nu}(x)| \de\H^{d-1}(y) = 0. $$
\end{enumerate}

To prove the proposition, we will use a covering argument to select cubes containing the majority of points in $S_G$ and $S_B$, for which the limits in the above list are approximately satisfied up to small relative error $\eta \ll 1$. Inside each of these nice cubes, we will modify the interface or leave it alone depending on the target energy. 

We begin by showing that for fixed $\eta > 0$, for $\H^{d-1}$-almost every $x\in \partial A$ there is $r_0(x,\eta) > 0$ such that for any $r<r_0$, there is a modification of $A$ given by $A_\eta$ for which $ A_\eta = A$ in a neighborhood of $\partial Q_{\nu(x)}(x,r)$ and the inequalities 
\begin{equation}\label{eqn:almostOptCube}
\begin{aligned}
\|\chi_{A_\eta} - \chi_A\|_{L^1(Q_{\nu(x)}(x,r))} \leq & \ \eta r^d , \\
\limsup_{\epsilon \to 0} \Per_\epsilon(A_\eta ; Q_{\nu(x)}(x,r))\leq & \int_{\partial A \cap Q_{\nu(x)}(x,r)} \ksr{\beta_\nu}\de\H^{d-1} + \eta  r^{d-1} 
\end{aligned}
\end{equation}
hold. As the idea is similar for $x\in S_G$, we focus on the more complicated situation where $x\in S_B.$ 

Without loss of generality, we suppose that $x = 0 \in S_B$, that $\nu(0) = -e_d$, and the relevant properties of \cref{hyp:graph} - \cref{hyp:goodOrientation} are satisfied.
Recalling the definition of $\ksr{\beta_\nu}$ in \labelcref{eq:beta_nu}, to construct a modification of $A$ in a small box centered at $0$, we break into cases: Either 
\begin{align*}
\beta_\nu (0) = \rho_0^{\pm\nu}(0) + \rho_1^{\pm\nu}(0)
\qquad
\text{or}
\qquad 
\beta_\nu (0 ) = \rho_0^{-\nu}(0) + \rho_1^{\nu}(0).
\end{align*}
In the first case, we shift the interface up or down to recover the optimal trace. In the second case, the energy is best when picking up the traces from either side of $A$, and consequently we will not modify $A,$ but must show this suffices.

\textbf{Case 1: $\beta_\nu (0) = \rho_0^{\pm\nu}(0) + \rho_1^{\pm\nu}(0)$.} Without loss of generality, we suppose that $\beta_\nu (0) = \rho_0^{-\nu}(0) + \rho_1^{-\nu}(0)$, which means that the energy is optimal slightly outside of $A$.

By \cref{hyp:graph}, the relative height satisfies $$h_{\partial A}(0,r) : = \sup_{y\in Q'(0,r)}\left\{\frac{g_0(y)}{r}\right\}\to 0 \quad \text{ as }r\to 0,$$ where $Q'$ denotes the $d-1$ dimensional cube. Consequently, also using that $\nabla g_0$ is continuous, we may assume that $r_0 = r(0,\eta)\ll 1$ is such that $$h_{\partial A}(0,r) + \|\nabla g_0\|_{L^\infty(Q'(0,r))} <\eta/2 \quad \text{ for }0<r<r_0.$$ Further by \cref{hyp:SBLeb}, we may take $r_0 \ll 1$ such that for $r<r_0$ we have
\begin{equation}\nonumber
\intbar_{Q^{+}_{e_d}(0,r)} |\rho_i(y) - \rho_i^{-\nu}(0)|\de y \leq \eta^2,
\end{equation}
which gives
\begin{equation}\nonumber
 \intbar_{\eta\ksr{r}}^{2\eta\ksr{r}}\intbar_{Q'(0,r)} |\rho_i(y',t) - \rho_i^{-\nu}(0)|\de y' d t \leq \eta.
 \end{equation}
Thus by \cref{lem:thetaFrac}, for $\theta \in (0,1)$, we may find a $\theta$-fraction of $a\in (\eta\ksr{r},2\eta\ksr{r})$ such that
\begin{equation}\label{eqn:slicedLeb}
 \intbar_{Q'(0,r)} |\rho_i(y',a) - \rho_i^{-\nu}(0)|\de y' \leq C(\theta)\eta.
 \end{equation}
Further, as $J_\rho$ is $\sigma$-finite with respect to $\mathcal{H}^{d-1}$ (as the absolute value of the jump is positive and integrable on this set), we can assume that for such choices of the value $a$ we also have $\mathcal{H}^{d-1} (J_\rho \cap (Q'(0,r) \times \{a\}) ) = 0,$ and thus 
\begin{equation}\label{eqn:coincidentTraces}
\rho_i^{\pm} =\rho_i \quad \text{ for }\H^{d-1}\text{-a.e. on }Q'(0,r) \times \{a\}.
\end{equation}
By \cref{lem:traceOnC1graph} and \labelcref{eqn:coincidentTraces}, for all $r' <r,$ we have 
\begin{equation}\label{eqn:hyperplaneLeb}
\begin{aligned}
& \lim_{\epsilon \to 0} \frac{1}{\epsilon} \int_0^\epsilon \intbar_{Q'(0,r')}\rho_0(y',a+t) \de y' \de t = \intbar _{Q'(0,r')}\rho_0(y',a) \de y', \\
&\lim_{\epsilon \to 0} \frac{1}{\epsilon} \int_0^\epsilon \intbar_{Q'(0,r')}\rho_1(y',a-t) \de y' \de t = \intbar _{Q'(0,r')}\rho_1(y',a) \de y'.
\end{aligned}
\end{equation}
\ksr{Now, we introduce a second small parameter $0< \delta<1$ (that can be taken equal to $\eta$), which we use to shrink the cube under consideration.}

Defining $A_\eta$ in the cube $Q(0,r)$ to be $A \cup (Q'(0,r(1-\delta))\times (-a,a)),$ the $L^1$ estimate of \labelcref{eqn:almostOptCube} follows immediately. 
\ksr{
By \cref{prop:minkowskiContent}, we have
$$\lim_{\eps \to 0}\mathcal{M}_\eps (\partial A_\eta \cap Q(0,r)\setminus Q(0,r(1-\delta))) =\H^{d-1}(\partial A_\eta \cap Q(0,r)\setminus Q(0,r(1-\delta))),$$
and one can additionally see that
\begin{align*}\nonumber
\H^{d-1}(\partial A_\eta \cap Q(0,r)\setminus Q(0,r(1-\delta)))= & \H^{d-1}(\partial A_\eta \cap Q(0,r)\setminus \overline{Q(0,r(1-\delta))}) +  \sum_{j}\mathcal{H}^{d-1} (F_{\pm j})  \\
\leq& \H^{d-1}(\partial A \cap Q(0,r)\setminus \overline{Q(0,r(1-\delta))}) + {C (h_{\partial A}+a)r^{d-1}} \\
\leq & \H^{d-1}({\partial A} \cap Q(0,r)\setminus Q(0,r(1-\delta))) + {C \eta r^{d-1}},
\end{align*} 
where $j$ sums over the faces $F_{\pm j} = \partial A_\eta \cap \{x \in \ksr{\overline{Q(0,r(1-\delta))}}:\pm x\cdot e_j = r(1-\delta)\}.$}


Consequently, we may apply \labelcref{eqn:hyperplaneLeb,eqn:slicedLeb}, the coarea formula \cite{ambrosio2000functions}, and \cref{hyp:massonSB,hyp:LebBeta} to find
\begin{equation}\label{bdd:cube1}
\begin{aligned}
&
\limsup_{\epsilon \to 0} \Per_\epsilon(A_\eta ; Q(0,r))    
\\
&\leq \int_{Q'(0,r(1-\delta))\times\{a\}}\left(\rho_0 +\rho_1 \right) \de\H^{d-1} + \left(\|\rho_0\|_{L^\infty} + \|\rho_1\|_{L^\infty}\right)\H^{d-1}(\partial A_\eta \cap Q(0,r)\setminus Q(0,r(1-\delta)))  \\
&\leq \left(\rho_0^{-\nu}(0)  + \rho_1^{-\nu}(0)\right) r^{d-1} + C(\theta)\eta r^{d-1} + C\ksr{\left( \delta (1+\|\nabla g_0\|_{L^\infty(Q'(0,r))}) + \eta\right)}r^{d-1} \\
&\leq  \left(\rho_0^{-\nu}(0)  + \rho_1^{-\nu}(0)\right) \H^{d-1}(\partial A \cap Q(x,r)) + C(\delta + \eta) r^{d-1} \\
&\leq  \int_{\partial A \cap Q(x,r)} \beta_\nu\de\H^{d-1} + C(\delta + \eta)r^{d-1} .
\end{aligned}
\end{equation}
As the constant $C>0$ arising in the previous inequality is independent of $x$ and $r$, taking $\delta  =\eta$, up to redefinition of $A_\eta$ by $A_{\eta/C}$, the proof of this case is concluded.

\textbf{Case 2: $\beta_\nu (0) = \rho_0^{-\nu}(0) + \rho_1^{\nu}(0)$.}
For $A = A_\eta$, the $L^1$ estimate of \labelcref{eqn:almostOptCube} is trivially satisfied. By \cref{lem:traceOnC1graph}, we have that
\begin{equation*}
\limsup_{\epsilon \to 0} \Per_\epsilon(A_\eta ; Q(0,r)) = \int_{\partial A \cap Q(x,r)} \left(\rho_0^{-\nu} +\rho_1^{\nu} \right) \de\H^{d-1}.
\end{equation*}
Assuming that the limits \cref{hyp:goodOrientation,hyp:LebBeta} are satisfied up to error $\eta$ for $r<r_0 = r(0,\eta)$, as in \labelcref{bdd:cube1} we can estimate
 \begin{equation}\nonumber
 \begin{aligned}
 \int_{\partial A \cap Q(x,r)} \left(\rho_0^{-\nu} +\rho_1^{\nu}\right) \de\H^{d-1} \leq & \,(\rho_0^{-\nu}(0) +\rho_1^{\nu}(0))\mathcal{H}^{d-1}(\partial A \cap Q(x,r)) + C\eta r^{d-1} \\ 
 = & \,\beta_\nu (0)\mathcal{H}^{d-1}(\partial A \cap Q(x,r)) + C\eta r^{d-1}  \\
 \leq & \, \int_{\partial A \cap Q(x,r)}\beta_\nu \de\mathcal{H}^{d-1} + C\eta r^{d-1}.
 \end{aligned}
 \end{equation}
 Once again, as $\eta>0$ is arbitrary, the proof of \labelcref{eqn:almostOptCube} is complete.

By the Morse measure covering theorem \cite{fonseca2007modern}, we may choose a finite collection of disjoint cubes $\{Q_i\}$ satisfying \labelcref{eqn:almostOptCube} compactly contained in $\Omega$ such that $$\H^{d-1}\left(\overline{\partial A \setminus \bigcup_i Q_i}^{\R^d}\right) = \H^{d-1}\left(\partial A \setminus \bigcup_i Q_i\right) <\eta \quad \text{ and }\quad \sum_i r_i^{d-1} \leq \ksr{\H^{d-1}(\partial A)}.$$ We define $A_\eta$ to be the set satisfying \labelcref{eqn:almostOptCube} in each of the finitely many disjoint cubes $Q_i$, and the same as $A$ outside of these cubes.
The $L^1$ estimate within the proposition statement is satisfied. For each point $x\in F:=\overline{\partial A_\eta \setminus \bigcup_i Q_i}^{\R^d}=\overline{\partial A \setminus \bigcup_i Q_i}^{\R^d}$, there is an open cube $C_x$ such that $F\cap \overline{C_x}^{\R^d}$ is given as the $C^1$-graph of a compact set up to rotation. Applying the Besicovitch covering theorem \cite{fonseca2007modern}, we may find a finite subset of $\{C_x\}$, given by $\{C_j\}$, covering $F$ such that each point $x$ belongs to $C_j$ for at most $C(d)$ many~$j$.

Noting that $$\lim_{\epsilon\to 0}\mathcal{M}_\epsilon(F\cap C_j ; \R^d)  = \mathcal{H}^{d-1}(F \cap \overline{C_j}^{\R^d}) =\mathcal{H}^{d-1}(F \cap C_j) $$ by \cref{prop:minkowskiContent} and using the properties of the selected $Q_i,$ we then estimate
\begin{equation}\nonumber
\begin{aligned}
\limsup_{\epsilon \to 0} \Per_\epsilon(A_\eta ; \Omega) \leq & \sum_i  \limsup_{\epsilon \to 0}\Per_\epsilon(A_\eta ; Q_i) + 2(\|\rho_0\|_{L^\infty} + \|\rho_1\|_{L^\infty}) \sum_j \limsup_{\epsilon \to 0} \mathcal{M}_\epsilon(F\cap C_j ; \R^d) \\
\leq & \int_{\partial A \cap\bigcup_i Q_i}\beta_\nu \de\H^{d-1} + C\eta\mathcal{H}^{d-1}(\partial A) + C(\rho_0,\rho_1)\sum_j \mathcal{H}^{d-1}(\ksr{F} \cap C_j) \\
\leq & \int_{\partial A}\beta_\nu \de\H^{d-1} +C(\rho_0,\rho_1, d, A)\eta.
\end{aligned}
\end{equation}
As $\eta > 0 $ is arbitrary, we conclude the proof of the proposition.
\end{proof}

We now complete the proof of the $\limsup$ bound in \cref{thm:limsup2}, and thereby the Gamma-convergence result of \cref{thm:gamma}.

\begin{proof}[Proof of \cref{thm:limsup2}]
As usual, we denote $\epsilon_k$ by $\epsilon.$ By \cref{thm:gammaCompact}, we may without loss of generality assume that $\chi_A \in BV(\Omega).$

\textbf{Step 1: Construction of an approximating ``smooth" set.}
For $\eta>0$, we apply the approximation result \ksr{\cite[Theorem C]{dePhilippis2017}} to find \ksr{$u\in SBV(\Omega)$} such that 
\begin{equation}\label{eqn:SBVsmoothApprox}
\begin{aligned}
&\|u-1/2\|_{L^\infty(\Omega)}\leq 1/2, \\
& \mathcal{H}^{d-1}(J_{u}\triangle \partial^*A)<\eta, \\
& J_{u}\subset M,  \text{ and } \mathcal{H}^{d-1}(M\setminus J_{u}) <\eta, \\
& M \subset\subset \Omega \text{ is a compact } C^1\text{-manifold with boundary (possibly empty)} ,\\
& u \in C^\infty(\Omega\setminus M), \\
& \|u - \chi_A\|_{BV(\Omega)} <\eta.
\end{aligned}
\end{equation}
\ksr{We note that the approximation is found by applying \cite[Theorem C]{dePhilippis2017} to $\chi_A-1/2$, and the $L^\infty$ bound follows from the comment at the top of page 372 therein (in fact, $\mathcal{H}^{d-1}(M\setminus J_u) = 0$ at this point). }

We will select a level-set of $u$ to approximate $A$, and to do this, we will need to know that the boundary of the level-set is well-behaved: away from $M$, it will suffice to apply Sard's theorem to see this is a $C^1$-manifold; however, it is possible the level-set boundary oscillates as it approaches $M$ and creates a large intersection. To ensure this doesn't happen, we begin by modifying $u$ so that 
\begin{equation}\label{eqn:smoothuptoman}
u \text{ is }C^1 \text{ up to the manifold }M \text{ and }\partial \Omega,
\end{equation} 
for which the precise meaning will become apparent.

To modify $u$ to satisfy \labelcref{eqn:smoothuptoman}, we locally use reflections to regularize. We remark that the trace is well defined on $M$, and up to a small extension of the manifold at $\partial M$, we can assume that $u^{+} = u^-$ in a neighborhood of $\partial M$ (meaning the manifold boundary). We modify $u$ as follows: For each $x \in M \cap \overline{\{u^+ \neq u^-\}}$ we choose $r_x>0$ such that $M\cap B(x,r_x)$ is a graph and $\overline{B(x,r)} \cap \partial M = \emptyset$. Consider a partition of unity $\{\psi_i\}$ with respect to a (finite) cover $\{B(x_i,r_{i})\}$ of $\overline{\{u^+ \neq u^-\}}.$ For each ball $B(x_i,r_i)$, define $M_i^\pm$ to be the ball intersected with the sub- or super-graph. In $M_i^\pm$, we can reflect and mollify $u\psi_i$. Choosing fine enough mollifications, restricting to $M_i^\pm$, and adding together the mollified functions (and $u(1-\sum_i\psi_i)$) provides the desired approximation satisfying \labelcref{eqn:smoothuptoman} and preserving the relations in \labelcref{eqn:SBVsmoothApprox}. Similarly, one may smoothly extend $u$ to $\R^d$ as $M \subset\subset \Omega$.

\ksr{\textbf{Substep 1.1: $\partial \{u>s\}$ approximates $\partial A$.}} Fixing $\theta \in (0,1)$, we show there is a $\theta$-fraction of $s\in (3/8,5/8)$ such that $A_s:=\{u>s\}$ is a good level-set approximating $A$, in the sense that, $\partial A_s$ is sufficiently regular (as for the next step) and
\begin{equation}\label{est:setAs}
\|\chi_{A_s} - \chi_A\|_{BV(\Omega)}\leq C(\theta)\eta.
\end{equation}
First, for any $s \in (3/8,5/8),$ we bound the $L^1$-norm by \labelcref{eqn:SBVsmoothApprox} as
\begin{equation}\label{eqn:setAsL1}
\begin{aligned}
\int_{\Omega} |\chi_{A_s} - \chi_A|\de x = & \, \mathcal{L}^d(\{u\leq s\}\cap A) + \mathcal{L}^d(\{u>s\}\cap A^c) \\
\leq & \left(\frac{1}{1-s} + \frac{1}{s}\right)\int_{\Omega} |u-\chi_A|\de x \leq C\eta.
\end{aligned}
\end{equation}
To control the gradient of $\chi_{A_s} - \chi_A$, first note that 
\begin{equation}\label{eqn:BVnormbreakdown}
    |D(\chi_{A_s} - \chi_A)|(\Omega)=\H^{d-1}(\partial^* \{u>s\} \triangle \partial^* A) +2\H^{d-1}(\partial^* \{u>s\} \cap \partial^* A \cap \{\nu_A \neq \nu_{A_s}\}),
\end{equation}
where we let $\nu_A$ and $\nu_{A_s}$ denote the measure-theoretic inner normals. To control the first term on the right-hand side of \labelcref{eqn:BVnormbreakdown}, we have
\begin{align}
\H^{d-1}(\partial^* \{u>s\} \triangle \partial^* A) =& \, \H^{d-1}(\partial^* \{u>s\} \setminus (J_u \cup \partial^* A)) \label{eqn:term1}\\
&+ \H^{d-1}(\partial^*\{u>s\}\cap J_u \setminus \partial^* A) \label{eqn:term2}\\
&+ \H^{d-1}(\partial^* A \setminus (J_u \cup \partial^* \{u>s\})) \label{eqn:term3} \\
&+ \H^{d-1}(\partial^* A \cap J_u \setminus \partial^* \{u>s\}). \label{eqn:term4}
\end{align}
We show that this symmetric difference can be chosen to be small. To control \labelcref{eqn:term1}, note that by the coarea formula (see \cite{ambrosio2000functions})
\begin{equation}\nonumber
\int_{-\infty}^\infty \mathcal{H}^{d-1}(\partial^*\{u>t\} \setminus (J_u \cup \partial^* A)) \de t  = \|\nabla u\|_{L^1(\Omega \setminus (J_u \cup \partial^* A))}\leq \|u - \chi_A\|_{BV(\Omega)}\leq \eta.
\end{equation}
Consequently, by \cref{lem:thetaFrac}, for a $\theta$-fraction of $s \in  (3/8,5/8)$, we have that
\begin{equation}\nonumber
\mathcal{H}^{d-1}(\partial^*\{u>s\} \setminus (J_u \cup \partial^* A)) \leq C(\theta) \eta.
\end{equation}
To estimate \labelcref{eqn:term2} and \labelcref{eqn:term3}, we have
\begin{equation}\label{eqn:pulloutSets}
\H^{d-1}(\partial^*\{u>s\}\cap J_u \setminus \partial^* A) + \H^{d-1}(\partial^* A \setminus (J_u \cup \partial^* \{u>s\})) \leq \H^{d-1}( J_u \triangle \partial^* A) \leq \eta.
\end{equation}
Finally, to control \labelcref{eqn:term4}, note that 
$$\partial^* A \cap J_u \setminus \partial^* \{u>s\} \subset  \partial^* A \cap J_u \cap \{|(u^+ - u^-) - 1|>1/4\}.$$ To see this, we prove the converse inclusion and suppose $x \in \partial^* A \cap J_u \cap \{|(u^+ - u^-) - 1|\leq 1/4\};$
\ksr{in particular, at such a point we have $u^+(x)\geq u^-(x) + 3/4$. Given the $L^\infty$ bound of \labelcref{eqn:SBVsmoothApprox}, we have $u^-(x) \geq 0$ and $1 \geq u^+(x)$.
Putting these facts together, we see that} $u^+\ksr{(x)} \geq 3/4 >1/4\geq u^-\ksr{(x)}$, showing that for $s \in  (3/8,5/8)$ we have $x\in \partial^*\{u>s\}$.
Thus, we can estimate \labelcref{eqn:term4} as follows:
\begin{equation}\label{est:term4} 
\begin{aligned}
\H^{d-1}(\partial^* A \cap J_u \setminus \partial^* \{u>s\}) \leq &\, \H^{d-1}(\partial^* A \cap J_u \cap \{|(u^+ - u^-) - 1|>1/4\}) \\
\leq &\, 4 \int_{\partial^* A \cap J_u}|(u^+ - u^-) - 1| \de\H^{d-1}\\
\leq &\, 4 \int_{\partial^* A \cap J_u}|(u-\chi_A)^+ - (u-\chi_A)^- | \de\H^{d-1}\\
\leq &\,  4\| u - \chi_A\|_{BV(\Omega)} \leq 4\eta.
\end{aligned}
\end{equation}
With the above estimates, up to redefinition of $\theta$, we have that
\begin{equation}\label{est:setAsGrad}
 \H^{d-1}(\partial^* \{u>s\} \triangle \partial^* A) <C(\theta)\eta
\end{equation}
for a $\theta$ portion of $s \in (3/8,5/8).$ 

We control the second right-hand side term in \labelcref{eqn:BVnormbreakdown} as follows. As in \labelcref{eqn:pulloutSets}, we have $\H^{d-1}(\partial^* \{u>s\} \cap \partial^* A \cap \{\nu_A \neq \nu_{\{u>s\}}\} \setminus J_u)\leq \H^{d-1}( J_u \triangle \partial^* A) < \eta$.  Then we note that
$$\int_{J_u\cap \partial^* A} |u^{\nu_A} - u^{-\nu_A} - 1|\de \H^{d-1} \leq \|u-  \chi_A\|_{BV(\Omega)}\leq \eta. $$
Consequently by \cref{lem:thetaFrac}, $u^{\nu_A} - u^{-\nu_A} > 3/4$ in $J_u\cap \partial^* A$ outside of a set with $\H^{d-1}$-measure less than $C\eta.$ As the normal $\nu_{A_s}$ coincides with the direction of positive change for $u$, for $s \in (3/8,5/8),$ we have that $\nu_A = \nu_{A_s}$ outside of a small set controlled by $\eta.$ In other words, we have that 
$$\H^d(\partial^* \{u>s\} \cap \partial^* A \cap \{\nu_A \neq \nu_{\{u>s\}}\}) < C\eta.$$
Putting this estimate together with the bounds \labelcref{eqn:setAsL1} and \labelcref{est:setAsGrad}, we conclude \labelcref{est:setAs}.

\ksr{
\textbf{Substep 1.2: Regularity of $\partial A_s$.} We \emph{claim} that for almost every choice of $s \in (3/8,5/8)$, $\overline{\partial A_s}^{\R^d}$ is contained in the compact image of a Lipschitz function $f:\R^d\to \overline \Omega$, i.e., $f(K) = \overline{\partial A_s}^{\R^d}$ for some $K\subset\subset \R^d$, and there is compact set $N \subset \overline\Omega$ with $\mathcal{H}^{d-1}(N) = 0$, such that for any point $x\in \overline{\partial A_s}^{\R^d} \setminus N$, there is a radius $r_x$ such that $\overline{\partial A_s}^{\R^d} \cap B(x,r_x) = \partial A_s \cap B(x,r_x)$ is a $C^1$ surface.

Note first by Sard's theorem, for almost every $s\in (0,1)$, the set $\partial A_s \cap \Omega \setminus \overline{\{u^+\neq u^-\}}$ is a $C^1$ surface away from $M\cup \partial \Omega$ (i.e., locally in $\Omega\setminus M$). We will show that the claim holds locally in $\overline{\{u^+\neq u^-\}}\cup \partial \Omega$, meaning that: for any $x_0 \in \overline{\{u^+\neq u^-\}}\cup \partial \Omega$, there is a radius $r>0$ such that, for almost every $s \in (0,1)$, $\overline{\partial A_s}^{\R^d} \cap \overline{B(x,r)}$  satisfies the claim, with the amendment that for $x\in (\overline{\partial A_s}^{\R^d} \setminus N)\cap B(x_0,r)$, there is $r_x>0$ such that $\overline{\partial A_s}^{\R^d} \cap B(x,r_x) = \partial A_s \cap B(x,r_x)$ is a $C^1$ surface for some compact set $N\subset \overline{\Omega}$ (this allows us to avoid proving anything on $\partial B(x_0,r)$). With the claim satisfied locally, a covering argument concludes the claim.

We will assume that $x_0\in \overline{\{u^+\neq u^-\}}$, as the case of $x_0\in \partial \Omega$ is simpler. Recall that we chose $M$ so that ${\rm dist}(\overline{\{u^+\neq u^-\}},\partial M)>0$, so that there is $r>0$ such that $B(x_0,2r) \cap M$ is a $C^1$ surface. Let $M^+$ and $M^-$ be the associated super-graph and sub-graph in $B(x_0,2r)$, respectively. By \labelcref{eqn:smoothuptoman}, $u|_{M^\pm}$ has a $C^1$ extension to $\R^d$, which we denote by $u^{{\rm ext}, \pm}$. Similarly, denoting the trace of $u$ from $M^\pm$ onto $M\cap B(x_0,2r)$ by $u^\pm,$ we have that $u^\pm$ belongs to $C^1(M \cap B(x_0,2r))$. Applying Sard's theorem three times, in $M^\pm$ and in $M \cap B(x_0,2r)$, we find that for almost every $s\in (0,1)$
$$ \partial \{u^{{\rm ext},\pm }>s\} \cap \overline{B(x_0,r)} \quad \text{ and } \quad \partial_M \{u^\pm >s\} \cap \overline{B(x_0,r)}$$
are $C^1$ surfaces of dimension $d-1$ and $d-2$, respectively,
and further 
$$\partial \{u^{{\rm ext},\pm }>s\} \cap M \cap  \overline{B(x_0,r)} = \partial_M \{u^\pm >s\} \cap \overline{B(x_0,r)}, $$
where $\partial_M$ denotes the boundary with respect to the topology relative to $M$.
Clearly we have
$$\partial A_s \cap \overline{B(x_0,r)}  \subset (M\cup \partial \{u^{{\rm ext},+ }>s\} \cup \partial \{u^{{\rm ext},-}>s\} ) \cap \overline{B(x_0,r)},$$
and taking $N : = (\partial_M \{u^+ >s\} \cup \partial_M \{u^- >s\}) \cap \overline{B(x_0,r)}$, we have that the claim is locally satisfied.
}

\textbf{Step 2: Good and bad parts of $\partial A_s$.}
We now fix $s \in (3/8,5/8)$ such that the previous step holds with regularity (as in Substep 1.2) and estimate \labelcref{est:setAs} for $A_s : = \{u>s\}.$ We construct open sets \ksr{$U_1$ and $U_2$ which cover $\overline{\partial A_s}^{\R^d}.$} The set $U_1$ will \ksr{only contain points} of $\partial A_s$ where it is a smooth manifold. The set $U_2$ will contain the $\H^{d-1}$-small \ksr{collection of points in $\overline{\Omega}$ for which $\overline{\partial A_s}^{\R^d}$ is not given by a smooth manifold, i.e., containing $N$.}

We choose $N \subset U_2 \subset  \{y \st \dist(y,N)\leq \delta\}$ for $0<\delta \ll 1$ such that
\begin{equation}\label{eqn:U2condies}
\mathcal{H}^{d-1}(\ksr{\overline{\partial A_s}^{\R^d}} \cap U_2) \leq \mathcal{H}^{d-1}(\ksr{\overline{\partial A_s}^{\R^d}} \cap \{y: \dist(y,N)\leq \delta\}) <\eta;
\end{equation}
this is possible as $\bigcap_{\delta>0} \{y \st \dist(y,N)\leq \delta\} = N$ and $\mathcal{H}^{d-1}(N) = 0.$

For every $x \in \partial A_s \setminus \ksr{U_2}$, there is a local neighborhood contained in $\Omega$ for which the boundary is a graph. Consequently, we may choose $U_1 \subset \subset \Omega$ to be an open set with smooth boundary such that the boundary is covered with \ksr{$\overline{\partial A_s}^{\R^d}\subset U_1\cup U_2$,} the boundary of $U_1$ is not charged, that is, $\mathcal{H}^{d-1}(\partial U_1 \cap \partial A_s)$, and $U_1$ is well separated from the sets arising from interface intersections with $\overline{U_1} \cap \ksr{N} = \emptyset$.

\textbf{Step 3: Near optimal approximation.} We will now construct a new set $A_\eta$ by modifying $A_s$ in $U_1$ while in \ksr{$U_2$}, we leave the surface unchanged. We will show \ksr{that this set $A_\eta$ satisfies}
\begin{equation}\label{eqn:AetaApprox}
\|\chi_{A_\eta} -\chi_{A}\|_{L^1(\Omega)}\leq \eta \quad \text{ and } \quad \limsup_{\epsilon \to 0} \Per_\epsilon(A_\eta ; \rho)\leq \Per(A ; \rho) +\eta .
\end{equation}
Within $U_i$ the constructed $A_\eta$ will leave the set $A_s$ unchanged in a neighborhood of $\partial U_i$. Consequently, for sufficiently small $\epsilon > 0$, we have
\begin{equation}\label{eqn:perDecomp}
\Per_\epsilon (A_\eta; \Omega) \leq \Per_\epsilon (A_\eta; U_1) + 2(\|\rho_0\|_{L^\infty} + \|\rho_1\|_{L^\infty})\mathcal{M}_\epsilon (\partial A_\eta \cap \ksr{U_2}),
\end{equation}
where $\mathcal{M}_\epsilon$ is defined in \labelcref{eqn:minkowskiContent}. 
By properties \labelcref{eqn:U2condies}, one can argue as at the end of the proof of \cref{prop:smoothRecovery} (with $\{C_j\}$) to find that
\begin{equation}\label{eqn:MinkContEst}
\limsup_{\epsilon\to 0} \mathcal{M}_\epsilon (\partial A_\eta \cap \ksr{U_2};\R^d) \leq C\eta.
\end{equation}
Further, by \cref{prop:smoothRecovery}, we can construct $A_\eta$ in $U_1$ such that $A_\eta = A_s$ in a neighborhood of $\partial U_1$,
\begin{equation}\label{eqn:Step2Claim}
\|\chi_{A_s} - \chi_{A_\eta}\|_{L^1(\Omega\cap U_1)}\leq \eta, \quad \text{ and } \quad \limsup_{\epsilon \to 0} \Per_\epsilon (A_\eta; U_1) \leq \int_{\partial A_s \cap U_1} \ksr{\beta_\nu} \de\mathcal{H}^{d-1} + \eta.
\end{equation}

With this, using \labelcref{est:setAs}, we see that the $L^1$ estimate in \labelcref{eqn:AetaApprox} immediately follows. To obtain the $\limsup$ inequality, we apply the decomposition \labelcref{eqn:perDecomp}, the estimate of the Minkowski content \labelcref{eqn:MinkContEst}, and the difference between $\partial A_s$ and $\partial^* A$ in \labelcref{eqn:BVnormbreakdown} to find
\begin{equation}\nonumber
\begin{aligned}
\limsup_{\epsilon \to 0} \Per_\epsilon(A_\eta ; \Omega) \leq & \int_{\partial A_s\cap U_1 } \ksr{\beta_\nu} \de \mathcal{H}^{d-1} + C\eta \\
\leq & \int_{\partial^* A} \ksr{\beta_\nu} \de \mathcal{H}^{d-1} + \|\ksr{\beta_\nu}\|_{L^\infty}\ksr{|D(\chi_{A_s} - \chi_A)|(\Omega)} +  C\eta\leq  \int_{\partial^* A} \ksr{\beta_\nu} \de \mathcal{H}^{d-1} +C\eta.
\end{aligned}
\end{equation}
Up to redefinition of $\eta$ to absorb the constant $C$, this concludes \labelcref{eqn:AetaApprox}, which by a diagonalization argument concludes the theorem. 
\end{proof}

For use in the next section, we highlight that the approximation introduced in the above proof may be used as a near optimal constant recovery sequence.

\begin{corollary}\label{cor:nearOptimal}
If the conditions of \cref{thm:gamma} hold true, then for all measurable sets $A\subset\R^d$ with $\Per(A;\bm{\rho})<\infty$ and for all $\eta>0$ there is a set $A_\eta$ with smooth boundary away from a finite union of $d-2$-dimensional manifolds having transverse intersection with the domain boundary, in the sense that 
\begin{equation}\nonumber
\H^{d-1}(\partial \Omega \cap \overline{\partial A_\eta}^{\R^d}) = 0,
\end{equation}
and such that
\begin{equation}\nonumber
\|\chi_{A_\eta} -\chi_{A}\|_{L^1(\Omega)}\leq \eta \quad \text{ and } \quad \limsup_{\epsilon \to 0} \Per_\epsilon(A_\eta ;\bm{\rho})\leq \Per(A ;\bm{\rho}) +\eta .
\end{equation}
\end{corollary}

\section{Applications}
\label{sec:applications}

Having proved our main statement \cref{thm:gamma}, we now turn towards applications.
We deduce Gamma-convergence of the total variation functional which is associated to $\Per_{\eps}(\cdot;\bm{\rho})$ in \cref{ssec:total_variation}, in \cref{ssec:AT} we discuss the asymptotic behavior of adversarial training as $\eps\to 0$, and in \cref{ssec:graph} we define discretizations of the nonlocal perimeter on random geometric graphs and prove their Gamma-convergence.

\subsection{Gamma-convergence of total variation}
\label{ssec:total_variation}

We define the nonlocal total variation of $u\in L^1(\Omega)$:
\begin{align}\label{eq:nonlocal_TV}
    \TV_\eps(u;\bm{\rho}) 
    :=
    \frac{1}{\eps}\int_\Omega \left(\esssup_{B(x,\eps)\cap\Omega} u - u(x)\right) \rho_0(x) \de x
    +
    \frac{1}{\eps}\int_\Omega \left(u(x) - \essinf_{B(x,\eps)\cap\Omega} u\right) \rho_1(x) \de x.
\end{align}
One can check easily (see also \cite[Proposition 3.13]{bungert2022geometry} for $L^\infty$-functions) that $\TV_\eps$ satisfies the generalized coarea formula
\begin{align}\nonumber
    \TV_\eps(u;\bm{\rho})= 
    \int_\R \Per_\eps(\{u\geq t\};\bm{\rho})\de t,\quad u\in L^1(\Omega).
\end{align}
For any integrable function, the measure of the set of values $t$ such that level-sets $\{u=t\}$ have positive mass is zero. 
Consequently, as definition \labelcref{eq:nonlocal_perimeter} is invariant under modification by null-sets, we may rewrite this as
\begin{align}\label{eq:coarea}
    \TV_\eps(u;\bm{\rho})= 
    \int_\R \Per_\eps(\{u> t\};\bm{\rho})\de t,\quad u\in L^1(\Omega).
\end{align}
This motivates us to define a limiting version of this total variation as 
\begin{align}\label{eq:local_TV}
     \TV(u;\bm{\rho}) := 
     \int_\R \Per(\{u> t\};\bm{\rho})\de t,\quad u\in L^1(\Omega),
\end{align}
which is identified as the Gamma-limit of $\TV_\eps$ in the following theorem.
\begin{theorem}
Under the conditions of \cref{thm:gamma} it holds
\begin{align*}
    \TV_\eps(\cdot;\bm{\rho}) \overset{\Gamma}{\longrightarrow} \TV(\cdot;\bm{\rho})
\end{align*}
as $\eps\to 0$ in the topology of $L^1(\Omega)$.
\end{theorem}
\begin{proof}
The result is a consequence of \cref{thm:gamma}, \labelcref{eq:coarea}, and \cite[Proposition 3.5]{chambolle2010continuous}---a generic Gamma-convergence result for functionals satisfying a coarea formula.
\end{proof}

We further provide a natural integral characterization of the limit energy \labelcref{eq:local_TV}, where we cannot directly argue via a density argument as $\beta$'s behavior on $d-1$-dimensional sets is not ``continuous" when approximated from the bulk.

\begin{proposition}
Under the conditions of \cref{thm:gamma} and for $\TV$ defined as in \labelcref{eq:local_TV}, the following representation holds
\begin{align}\nonumber
    \TV(u;\bm{\rho}) 
    =
    \int_\Omega \beta\left(\frac{D u}{\abs{D u}};\bm{\rho}\right)\de\abs{D u},\qquad u \in BV(\Omega).
\end{align}
\end{proposition}
\begin{proof}
For $u$ fixed, we may define $\beta : = \beta\left(\frac{D u}{\abs{D u}};\bm{\rho}\right)$ for $\H^{d-1}$-almost every point and treat it as a fixed function. Given the properties of the jump set \cite[Section 3.6]{ambrosio2000functions}, it follows that $\beta$ has an $\H^{d-1}$-equivalent Borel representative.
We can then rewrite the equality in the proposition as
\begin{equation}\label{eqn:almostTVrep}
\int_\R \left[\int_{\partial^*\{u>t\}\cap \Omega} \beta  \de \H^{d-1}\right]\de t =  \int_\Omega \beta\de\abs{D u},
\end{equation}
where $\beta$ is a generic positive, bounded Borel measurable function. By a standard approximation argument, \labelcref{eqn:almostTVrep} will follow if we show that it holds for $\beta := \chi_{A}$ for any Borel measurable subset $A\subset \Omega.$

To extend to a generic Borel subset, define the class $\mathcal{S}:=\{A\subset \Omega: \labelcref{eqn:almostTVrep} \text{ holds for }\beta:=\chi_A\}.$ For $A$, an open subset of $\Omega$, and $\beta := \chi_A$, \labelcref{eqn:almostTVrep} reduces to the standard coarea formula for $BV$ functions \cite{ambrosio2000functions}. Consequently $\mathcal{S}$ contains all open subsets.
Noting that $\mathcal{S}$ satisfies the hypothesis of the monotone class (or $\pi-\lambda$) theorem \cite[Theorem 1.4]{EvansGariepy}, it follows that $\mathcal{S}$ contains all Borel subsets, concluding the proposition. 
\end{proof}

\subsection{Asymptotics of adversarial training}
\label{ssec:AT}

In this section we would like to apply our Gamma-convergence results to adversarial training \labelcref{eq:AT}.
For this we let $\mu\in\mathcal M(\Omega\times\{0,1\})$ be the measure characterized through
\begin{align}
    \mu(\cdot\times\{0\}) := \rho_0,
    \qquad
    \text{and}
    \qquad
    \mu(\cdot\times\{1\}) := \rho_1
\end{align}
and consider the following version of adversarial training
\begin{align}\label{eq:AT_Omega}
    \inf_{A\in\mathfrak B(\Omega)}\E_{(x,y)\sim\mu}\left[\sup_{\tilde x\in B(x,\eps)\cap\Omega}\abs{\chi_A(\tilde x)-y}\right],
\end{align}
which arises from \labelcref{eq:AT} by choosing $\X=\Omega$ equipped with the Euclidean distance and $\C := \{\chi_A \st A \in \mathfrak B(\Omega)\}$ as the collection of characteristic functions of Borel sets.
As proved in \cite{bungert2022geometry} the problem can equivalently be reformulated as
\begin{align}\label{eq:TV_regularization_equiv}
    \inf_{A\in\mathfrak B(\Omega)} \E_{(x,y)\sim\mu}[\abs{\chi_A(x)-y}] + \eps \Per_\eps(A;\bm{\rho}),
\end{align}
where $\Per_\eps(A;\bm{\rho})$ is the nonlocal perimeter \labelcref{eq:nonlocal_perimeter} for which we proved Gamma-convergence in \cref{sec:Gamma-convergence}.

Before we turn to convergence of minimizers of this problem, we first discuss an alternative and simpler model for adversarial training which arises from fixing the regularization parameter in front of the nonlocal perimeter in \labelcref{eq:TV_regularization_equiv} to $\alpha>0$:
\begin{align}\label{eq:TV_regularization}
    \inf_{A\subset\Omega} \E_{(x,y)\sim\mu}[\abs{\chi_A(x)-y}] + \alpha \Per_\eps(A;\bm{\rho}).
\end{align}
Unless for $\alpha=\eps$ this problem is not equivalent to adversarial training anymore but it can be interpreted as an affine combination of the non-adversarial and the adversarial risk:
\begin{align}
    \labelcref{eq:TV_regularization}
    =
    \inf_{A\subset\Omega}
    \frac{\alpha}{\eps}
    \E_{(x,y)\sim\mu}\left[\sup_{\tilde x\in B(x,\eps)\cap\Omega}\abs{\chi_A(\tilde x)-y}\right]
    +
    \left(1-\frac{\alpha}{\eps}\right)
    \E_{(x,y)\sim\mu}\left[\abs{\chi_A(x)-y}\right].
\end{align}
Gamma-convergence as $\eps\to 0$ is an easy consequence of \cref{thm:gamma} since \labelcref{eq:TV_regularization} is a continuous perturbation of a Gamma-converging sequence of functionals.
\begin{corollary}
Under the conditions of \cref{thm:gamma}, the functionals
\begin{align*}
    A \mapsto 
    \E_{(x,y)\sim\mu}[\abs{\chi_A(x)-y}] + \alpha \Per_\eps(A;\bm{\rho})
\end{align*}
Gamma-converge in the $L^1(\Omega)$ topology as $\eps\to 0$ to the functional
\begin{align*}
    A \mapsto 
    \E_{(x,y)\sim\mu}[\abs{\chi_A(x)-y}] + \alpha \Per(A;\bm{\rho}).
\end{align*}
\end{corollary}
Let us now continue the discussion of the original adversarial training problem \labelcref{eq:AT_Omega} (or equivalently \labelcref{eq:TV_regularization}).
In order to preserve the regularizing effect of the perimeter, it is natural to take the approach of \cite{burger2023variational,belenkin2022note}, developed in the context of Tikhonov regularization for inverse problems, and to consider the rescaled functional
\begin{align}
    J_\eps(A) := \frac{\E_{(x,y)\sim\mu}[\abs{\chi_A(x)-y}]-\inf_{B\in\mathfrak B(\Omega)}\E_{(x,y)\sim\mu}[\abs{\chi_B(x)-y}]}{\eps} + \Per_\eps(A;\bm{\rho}).
\end{align}
Obviously, this functional has the same minimizers as the original problem \labelcref{eq:AT_Omega,eq:TV_regularization_equiv}. 
Judging from the results in \cite{burger2023variational,belenkin2022note} one might hope that the Gamma-limit of $J_\eps$ is the functional
\begin{align}
    J(A) :=
    \begin{dcases}
    \Per(A;\bm{\rho}),\quad&\text{if}\quad A \in \argmin_{B\in\mathfrak B(\Omega)}\E_{(x,y)\sim\mu}[\abs{\chi_B(x)-y}],\\
    \infty,\quad&\text{else}.
    \end{dcases}
\end{align}
Indeed the $\liminf$ inequality is trivially satisfied:
\begin{lemma}\label{lem:liminf_AT}
Under the conditions of \cref{thm:gamma} it holds for any sequence of measurable sets $(A_k)_{k\in\N}\subset\Omega$ and any sequence $(\eps_k)_{k\in\N}\subset(0,\infty)$ with $\chi_{A_k}\to\chi_A$ in $L^1(\Omega)$ and $\lim_{k\to\infty}\eps_k=0$ that  
\begin{align*}
    J(A) 
    \leq
    \liminf_{k\to\infty}
    J_{\eps_k}(A_k).
\end{align*}
\end{lemma}
\begin{proof}
If we assume that $\alpha:=\E_{(x,y)\sim\mu}[\abs{\chi_A(x)-y}]-\inf_{B\in\mathfrak B(\Omega)}\E_{(x,y)\sim\mu}[\abs{\chi_B(x)-y}]>0$ (and hence $J(A)=\infty$) then we have by continuity
\begin{align*}
    \alpha
    =
    \lim_{k\to\infty}
    \E_{(x,y)\sim\mu}[\abs{\chi_{A_k}(x)-y}]
    -
    \inf_{B\in\mathfrak B(\Omega)}\E_{(x,y)\sim\mu}[\abs{\chi_B(x)-y}]
    .
\end{align*}
Hence, we see 
\begin{align*}
    \liminf_{k\to\infty}J_{\eps_k}(A_k)
    &=
    \liminf_{k\to\infty}
    \frac{\E_{(x,y)\sim\mu}[\abs{\chi_{A_k}(x)-y}]-\inf_{B\in\mathfrak B(\Omega)}\E_{(x,y)\sim\mu}[\abs{\chi_B(x)-y}]}{\eps_k} + \Per_{\eps_k}(A_k;\bm{\rho})
    \\
    &\geq 
    \alpha\liminf_{k\to\infty}\frac{1}{\eps_k}
    =
    \infty
    =
    J(A).
\end{align*}
In the other case if $\alpha=0$ we use the $\liminf$ inequality from \cref{thm:gamma} to find
\begin{align*}
    \liminf_{k\to\infty}J_{\eps_k}(A_k)
    &=
    \liminf_{k\to\infty}
    \frac{\E_{(x,y)\sim\mu}[\abs{\chi_{A_k}(x)-y}]-\inf_{B\in\mathfrak B(\Omega)}\E_{(x,y)\sim\mu}[\abs{\chi_B(x)-y}]}{\eps_k} + \Per_{\eps_k}(A_k;\bm{\rho})
    \\
    &\geq 
    \Per(A;\bm{\rho})
    =
    J(A)
\end{align*}
since the first term in $J_{\eps_k}$ is non-negative.
\end{proof}
The $\limsup$ inequality is non-trivial and potentially even false.
Letting $(A_k)_{k\in\N}$ be a recovery sequence for the perimeter $\Per(A;\bm{\rho})$ of $A\in\argmin_{B\in\mathfrak B(\Omega)}\E_{(x,y)\sim\mu}[\abs{\chi_B(x)-y}]$ one has
\begin{align*}
    \limsup_{k\to\infty} J_{\eps_k}(A_k) 
    &=
    \limsup_{k\to\infty}
    \frac{\E_{(x,y)\sim\mu}[\abs{\chi_{A_k}(x)-y}]-\inf_{B\in\mathfrak B(\Omega)}\E_{(x,y)\sim\mu}[\abs{\chi_B(x)-y}]}{\eps_k} + \Per_{\eps_k}(A_k;\bm{\rho})
    \\
    &\leq 
    \limsup_{k\to\infty}
    \frac{\E_{(x,y)\sim\mu}[\abs{\chi_{A_k}(x)-y}]-\inf_{B\in\mathfrak B(\Omega)}\E_{(x,y)\sim\mu}[\abs{\chi_B(x)-y}]}{\eps_k}
    +
    \Per(A;\bm{\rho}).
\end{align*}
It obviously holds that
\begin{align*}
    \limsup_{k\to\infty}
    \frac{\E_{(x,y)\sim\mu}[\abs{\chi_{A}(x)-y}]-\inf_{B\in\mathfrak B(\Omega)}\E_{(x,y)\sim\mu}[\abs{\chi_B(x)-y}]}{\eps_k}
    +
    \Per(A;\bm{\rho})
    = 
    J(A),
\end{align*}
and hence, for the $\limsup$ inequality to be satisfied, we would need to make sure that
\begin{align*}
    \limsup_{k\to\infty}
    \frac{\E_{(x,y)\sim\mu}[\abs{\chi_{A}(x)-\chi_{A_k}(x)}]}{\eps_k}
    =
    \limsup_{k\to\infty}
    \frac{\int_\Omega\abs{\chi_{A}-\chi_{A_k}}(\rho_0+\rho_1)\de x}{\eps_k}
    =
    0.
\end{align*}
This requires that the recovery sequences converges sufficiently fast to $A$ in $L^1(\Omega)$, namely
\begin{align}\label{eq:fast_cvgc}
    \L^d(A \triangle A_k) = o(\eps_k),
\end{align}
and is not obvious from our proof of \cref{thm:limsup2}.
Even for smooth densities $\rho_0,\rho_1$ where the construction of the recovery sequences is much simpler, \labelcref{eq:fast_cvgc} is not obvious.

On the other hand, condition \labelcref{eq:fast_cvgc} for the validity of the $\limsup$ inequality only has to be satisfied for the minimizers (so-called Bayes classifiers) of the unregularized problem $\inf_{B\in\mathfrak B(\Omega)}\E_{(x,y)\sim\mu}[\abs{\chi_B(x)-y}]$ which have finite weighted perimeter.

This motivates to assume a so-called ``source condition'', demanding some regularity on these Bayes classifiers. 
In the field of inverse problems source conditions are well-studied and known to be necessary for proving convergence of variational regularization schemes \cite{benning2018modern,burger2023variational}.
Our first source condition---referred to as strong source condition---takes the following form:
\begin{align}\label{eq:source_condition_strong}\tag{sSC}
    \begin{split} 
        &\text{All Bayes classifiers
        $A^\dagger\in\argmin_{B\in\mathfrak B(\Omega)}\E_{(x,y)\sim\mu}[\abs{\chi_{B}(x)-y}]$ with $\Per(A^\dagger;\bm{\rho})<\infty$}
        \\
        &\text{possess a recovery sequence satisfying \labelcref{eq:fast_cvgc}.}
    \end{split}
\end{align}
Note that a Bayes classifier $A^\dagger$ admits a recovery sequence satisfying \labelcref{eq:fast_cvgc}, for instance, if $\partial A^\dagger$ is sufficiently smooth and the densities $\rho_0$ and $\rho_1$ are continuous.
In this case the constant sequence, which trivially satisfies \labelcref{eq:fast_cvgc}, recovers.
Under this strong condition we have proved the following  Gamma-convergence result:
\begin{proposition}[Conditional Gamma-convergence]
Under the conditions of \cref{thm:gamma} and assuming \labelcref{eq:source_condition_strong} it holds that
\begin{align*}
    J_\eps \overset{\Gamma}{\longrightarrow} J
\end{align*}
as $\eps\to 0$ in the $L^1(\Omega)$ topology.
\end{proposition}
In fact, a \ksr{substantially} weaker source condition suffices for compactness of solutions as $\eps\to 0$:
\begin{align}\label{eq:source_condition_weak}\tag{wSC}
    \begin{split} 
        &\text{There exists a Bayes classifier $A^\dagger\in\argmin_{B\in\mathfrak B(\Omega)}\E_{(x,y)\sim\mu}[\abs{\chi_{B}(x)-y}]$ with $\Per(A^\dagger;\bm{\rho})<\infty$}
        \\
        &\text{which possesses a recovery sequence satisfying \labelcref{eq:fast_cvgc}.}
    \end{split}
\end{align}
We get the following compactness statement assuming validity of this source condition.
\begin{proposition}[Conditional compactness]\label{prop:compactAT}
Under the conditions of \cref{thm:gammaCompact,thm:gamma} and assuming the source condition \labelcref{eq:source_condition_weak}, any sequence of solutions to \labelcref{eq:AT_Omega} is precompact in $L^1(\Omega)$ as $\eps\to 0$.
\end{proposition}
\begin{proof}
Let us take a sequence of solutions $A_\eps$ of \labelcref{eq:AT_Omega} for $\eps\to 0$.
Furthermore let $A_\eps^\dagger$ denote a recovery sequence for the Bayes classifier $A^\dagger$ which satisfies \labelcref{eq:source_condition_weak}.
Using the minimization property of $A_\eps$ it holds
\begin{align*}
    \E_{(x,y)\sim\mu}[\abs{\chi_{A_\eps}(x)-y}] + \eps \Per_\eps(A_\eps;\bm{\rho})
    \leq 
    \E_{(x,y)\sim\mu}[\vert\chi_{A_\eps^\dagger}(x)-y\vert] + \eps \Per_\eps(A_\eps^\dagger;\bm{\rho}).
\end{align*}
Subtracting the Bayes risk and rescaling by $\eps$, we have
\begin{align*}
    \begin{split}
    &\frac{\E_{(x,y)\sim\mu}[\abs{\chi_{A_\eps}(x)-y}] - \E_{(x,y)\sim\mu}[\abs{\chi_{A^\dagger}(x)-y}]}{\eps} + \Per_\eps(A_\eps;\bm{\rho})
    \\
    &\qquad
    \leq 
    \frac{\E_{(x,y)\sim\mu}[\vert\chi_{A_\eps^\dagger}(x)-y\vert]- \E_{(x,y)\sim\mu}[\abs{\chi_{A^\dagger}(x)-y}]}{\eps}
     + \Per_\eps(A_\eps^\dagger;\bm{\rho}).
     \end{split}
\end{align*}
Using that the leftmost term is non-negative, taking the $\limsup$, and using \labelcref{eq:source_condition_weak} yields
\begin{align*}
    \limsup_{\eps\to 0}\Per_\eps(A_\eps;\bm{\rho}) 
    \leq
    \limsup_{\eps\to 0}\Per_\eps(A_\eps^\dagger;\bm{\rho})
    \leq 
    \Per(A^\dagger;\bm{\rho})<\infty.
\end{align*}
Hence, we can apply \cref{thm:gammaCompact} and conclude.
\end{proof}
So far we have introduced the strong source condition \labelcref{eq:source_condition_strong} for proving Gamma-convergence and the weak one \labelcref{eq:source_condition_weak} for showing compactness.
As it turns out, for proving \cref{thm:convergence_AT}, concerning convergence of minimizers of adversarial training \labelcref{eq:AT_Omega}, it suffices to assume the source condition \labelcref{eq:source_condition} which is in the middle \ksr{but only slightly stronger than \labelcref{eq:source_condition_weak}}, i.e.,
$$\text{\labelcref{eq:source_condition_strong} $\implies$ \labelcref{eq:source_condition} $\implies$ \labelcref{eq:source_condition_weak}.}$$
This condition is the following:
\begin{align}\label{eq:source_condition}\tag{SC}
    \begin{split} 
        &\text{There exists a perimeter minimal Bayes classifier}
        \\
        &\qquad A^\dagger\in\argmin\left\lbrace \Per(A;\bm{\rho}) \st A \in \argmin_{B\in\mathfrak B(\Omega)}\E_{(x,y)\sim\mu}[\abs{\chi_{B}(x)-y}]\right\rbrace
        \\
        &\text{with $\Per(A^\dagger;\bm{\rho})<\infty$ which possesses a recovery sequence satisfying \labelcref{eq:fast_cvgc}.}
    \end{split}
\end{align}
Under this condition we can prove our last main result.
\begin{proof}[Proof of \cref{thm:convergence_AT}.]
Since \labelcref{eq:source_condition} implies \labelcref{eq:source_condition_weak}, by \cref{prop:compactAT} up to a subsequence it holds $A_\eps\to A$ in $L^1(\Omega)$.
Let $A_\eps^\dagger$ denote a recovery sequence for $A^\dagger$ satisfying \labelcref{eq:source_condition}.
\cref{lem:liminf_AT} implies 
\begin{align*}
    J(A)\leq \liminf_{\eps\to 0} J_\eps(A_\eps) \leq
    \limsup_{\eps\to 0} J_\eps(A_\eps^\dagger)
    \leq 
    \Per(A^\dagger;\bm{\rho}).
\end{align*}
Since $\Per(A^\dagger;\bm{\rho})<\infty$ we get that $J(A)<\infty$ and hence $A\in\argmin_{B\in\mathfrak B(\Omega)}\E_{(x,y)\sim\mu}[\abs{\chi_B(x)-y}]$ and $\Per(A;\bm{\rho})=J(A)\leq\Per(A^\dagger;\bm{\rho})$.
This shows that $A$ is a minimizer of the problem in \labelcref{eq:perimeter_minimal_solution}.
\end{proof}

\subsection{Gamma-convergence of graph discretizations}
\label{ssec:graph}

In this section we discuss a discretization of the nonlocal perimeter $\Per_\eps(\cdot;\bm{\rho})$ on a random geometric graph and prove Gamma-convergence in a suitable topology.
For this, let $G_n=(X_n,W_n)$ \ksr{for $n\in\N$} be a weighted \ksr{random geometric} graph with vertex set $X_n$ and weights $W_n$.
This means that the vertex set $X_n := \{x_1,\dots,x_n\}\subset\Omega$ \ksr{is a collection of independent and identically distributed (\textit{i.i.d.}) random variables with law} $\rho := \rho_0 + \rho_1$, and the weights are defined as $W_n(x,y) := \ksr{\phi}(\abs{x-y}/\eps_n)$ for some parameter $\eps_n>0$ and $\ksr{\phi}(t):=\chi_{\ksr{[0,1]}}(t)$ for \ksr{$t\in\R$}.

We can identify the graph vertices with their empirical measure $\nu_n := \frac{1}{n}\sum_{i=1}^n\delta_{x_i}$.
Thanks to \cite{trillos2015rate} (see also \cite[Theorem 2.5]{GarcSlep15}) and the fact that $\rho$ is strictly positive on $\Omega$, \ksr{with probability one} (almost surely) there exists maps $T_n : \Omega\to\Omega$ such that 
\begin{subequations}\label{eq:prop_transport_map}
    \begin{align}
        (T_n)_\sharp \rho &= \nu_n, \\
        \limsup_{n\to\infty}\frac{\norm{T_n-\operatorname{id}}_{L^\infty(\Omega)}}{\delta_n} &< \infty.
    \end{align}
\end{subequations}
Here the quantity $\delta_n>0$ is given by
\begin{align*}
    \delta_n :=
    \begin{dcases}
    \frac{(\log n)^\frac{3}{4}}{n^\frac{1}{2}},
    \quad&\text{if }d=2,\\
    \left(\frac{\log n}{n}\right)^\frac{1}{d},\quad&\text{if }d\geq 3,
    \end{dcases}
\end{align*}
which for $d\geq 3$ coincides with the asymptotic connectivity threshold for random geometric graphs.
We can use these transport maps to consider the measures
\begin{align}\label{eq:label_measures}
    \nu_n^i := (T_n)_\sharp \rho_i,\quad i \in \{0,1\}
\end{align}
which are the empirical measures of the graph points that are associated with the $i$-th density $\rho_i$.

We define a graph discretization of $\Per_\eps(A;\bm{\rho})$ for $A\subset X_n$ as
\begin{align*}
    E_n(A)
    :=
    \frac{1}{\eps_n}
    \nu_n^0(\left\lbrace
    x \in A^c \st \dist(x,A) < \eps_n
    \right\rbrace)
    +
    \frac{1}{\eps_n}
    \nu_n^1(\left\lbrace
    x \in A \st \dist(x,A^c) < \eps_n
    \right\rbrace)
\end{align*}
which effectively counts the number of points in an exterior strip around $A$ carrying the label $0$ and the number of points in an interior strip in $A$ carrying the label $1$.
\ksr{Note that, although the complement $A^c$ of a subset $A$ of the graph vertices $X_n$ is not a subset of $X_n$ anymore, the empirical measure $\nu_n^1$ only considers points in $A^c\cap X_n$ so one can just as well replace $A^c$ by $X_n \setminus A$.}
Note also that this graph model using \labelcref{eq:label_measures} assumes that the label distribution is performed according to the ground truth distributions $\rho_0$ and $\rho_1$. 
One can also treat more general labeling models such that \labelcref{eq:label_measures} is asymptotically satisfied as $n\to\infty$, but for the sake of simplicity we limit the discussion to the model above.

Using the weight function $W_n(x,y)$ we can equivalently express $E_n(A)$ as
\begin{align*}
    E_n(A)
    &=
    \frac{1}{\eps_n}
    \int_\Omega 
    \left(\max_{x \in X_n}W_n(x,y)\chi_{A}(x) - \chi_{A}(y)\right)
    \de\nu_n^0(y)
    \\
    &\qquad
    +
    \frac{1}{\eps_n}
    \int_\Omega 
    \ksr{\left(\max_{x \in X_n}W_n(x,y)\chi_{A^c}(x) - \chi_{A^c}(y)\right)}
    \de\nu_n^1(y)
    \\
    &=
    \frac{1}{\eps_n}
    \int_\Omega 
    \left(\max_{x \in X_n}\ksr{\phi}\left({\abs{x-y}}/{\eps_n}\right)\chi_{A}(x) - \chi_{A}(y)\right)
    \de\nu_n^0(y)
    \\
    &\qquad
    +
    \frac{1}{\eps_n}
    \int_\Omega 
    \ksr{\left(\max_{x \in X_n}\ksr{\phi}\left({\abs{x-y}}/{\eps_n}\right)\chi_{A^c}(x) - \chi_{A^c}(y)\right)}
    \de\nu_n^1(y).
\end{align*}
This graph perimeter functional combines elements of the graph perimeter studied in \cite{GarcSlep16} and the graph Lipschitz constant studied in \cite{roith2022continuum}.
Correspondingly, also the following Gamma-convergence proof bears similarities with both of these works. 
For proving Gamma-convergence of these graph perimeters to the continuum perimeter we employ the $TL^p$-framework, developed in \cite{GarcSlep16}.
Here, we do not go into too much detail regarding the definition and the properties of these metric spaces. 
We just define the space as the set of pairs of $L^p$ functions and measures $TL^p(\Omega) := \{(f,\mu) \st \mu \in \P(\Omega),\; f\in L^p(\Omega,\mu)\}$, where $\P(\Omega)$ is the set of probability measures on $\Omega$.
We highlight \cite[Proposition 3.12, 4.]{GarcSlep16} which says that in our specific situation with $\rho$ being a strictly positive absolutely continuous measure, convergence of $(u_n,\nu_n)\to (u,\rho)$ in the topology of $TL^p(\Omega)$ is equivalent to $u_n \circ T_n \to u$ in $L^p(\Omega)$ for the maps $T_n$ satisfying \labelcref{eq:prop_transport_map}.
\ksr{Furthermore, we would like to emphasize that the functionals $E_n$ are random variables since they depend on the given realization of the random variables which constitute the vertices $X_n$ of the graph. 
Still it is possible to prove Gamma-convergence of these functionals with probability one, meaning that Gamma-convergence might be violated only for a set of graph realizations which have zero probability.
We refer the interested reader to \cite[Definition 2.11]{GarcSlep16} for precise definitions.}

The following is the main result of this section and asserts Gamma-convergence of the functionals $E_n$ to the Gamma-limit from \cref{thm:gamma}.

\begin{theorem}\label{thm:gamma_graph}
Let the assumptions of \cref{thm:gamma} be satisfied.
If $\eps_n>0$ satisfies
\begin{align*}
    \lim_{n\to \infty}\eps_n = 0
    \qquad
    \text{and}
    \qquad
    \lim_{n\to \infty}\frac{\delta_n}{\eps_n} = 0,
\end{align*}
then with probability one it holds
\begin{align*}
    E_n
    \overset{\Gamma}{\longrightarrow}
    \Per(\cdot;\bm{\rho})
\end{align*}
as $n\to\infty$ in the $TL^1(\Omega)$ topology, and the following compactness property holds:
\begin{align*}
    \limsup_{n\to \infty}E_n(A_{n}) < \infty
    \quad
    \implies
    \quad
    (A_n)_{n\in\N} 
    \text{ is precompact in $TL^1(\Omega)$}.
\end{align*}
\end{theorem}
\begin{proof}
The result is proved in \cref{lem:graph_liminf,lem:graph_limsup} which prove the liminf and the limsup inequalities. 
The proof of the compactness statement is implicitly contained in the proof of \cref{lem:graph_liminf}.
\end{proof}

\begin{lemma}[Discrete liminf inequality]\label{lem:graph_liminf}
Under the conditions of \cref{thm:gamma_graph} for any sequence $(A_n)\subset \ksr{X_n}$ such that $(\chi_{A_n},\nu_n)\to (\chi_A,\rho)$ in $TL^1(\Omega)$ it holds with probability one that
\begin{align*}
    \Per(A;\bm{\rho}) \leq \liminf_{n\to\infty} E_n(A_n).
\end{align*}
\end{lemma}
\begin{proof} 
We can perform a change of variables and use \cite[Lemma 2]{roith2022continuum} to obtain
\begin{align*}
    E_n(A_n)
    &:=
    \frac{1}{\eps_n}
    \int_\Omega 
    \left(\max_{x \in X_n}\ksr{\phi}\left({\abs{x-y}}/{\eps_n}\right)\chi_{A_n}(x) - \chi_{A_n}(y)\right)
    \de\nu_n^0(y)
    \\
    &\qquad
    +
    \frac{1}{\eps_n}
    \int_\Omega 
    \ksr{
    \left(\max_{x \in X_n}\ksr{\phi}\left({\abs{x-y}}/{\eps_n}\right)\chi_{A_n^c}(x) - \chi_{A_n^c}(y)\right)
    }
    \de\nu_n^1(y)
    \\
    &=
    \frac{1}{\eps_n}
    \int_\Omega 
    \left(\esssup_{x \in \Omega}\ksr{\phi}\left({\abs{T_n(x)-T_n(y)}}/{\eps_n}\right)\chi_{A_n}(T_n(x)) - \chi_{A_n}(T_n(y))\right)
    \rho_0(y)\de y
    \\
    &\qquad
    +
    \frac{1}{\eps_n}
    \int_\Omega 
    \ksr{
    \left(\esssup_{x \in \Omega}\ksr{\phi}\left({\abs{T_n(x)-T_n(y)}}/{\eps_n}\right)\chi_{A_n^c}(T_n(x)) - \chi_{A_n^c}(T_n(y))\right)
    }
    \rho_1(y)\de y
\end{align*}
Let us \ksr{consider points $x,y\in \Omega$ with} $\ksr{\phi}\left({\abs{T_n(x)-T_n(y)}}/{\eps_n}\right) = 0$ which is equivalent to $\abs{T_n(x)-T_n(y)}>\eps_n$.
Then it holds that 
\begin{align}
    \nonumber 
    \abs{x-y} 
    &\geq 
    \abs{T_n(x)-T_n(y)}
    -
    \abs{T_n(x)-x}
    -
    \abs{T_n(y)-y}
    \\
    \nonumber 
    &\geq 
    \eps_n - 2 \norm{T_n-\operatorname{id}}_{L^\infty(\Omega)}
    \\
    \label{eq:def_tilde_eps}
    &=
    \eps_n\left(1-2\frac{\norm{T_n-\operatorname{id}}_{L^\infty(\Omega)}}{\eps_n}\right)
    =:
    \Tilde{\eps}_n
\end{align}
\ksr{which is equivalent to $\phi\left(\abs{x-y}/\tilde\eps_n\right)=0$.}
Hence, \ksr{we have proved that $\phi\left(\abs{T_n(x)-T_n(y)}/\eps_n\right)\geq \phi\left(\abs{x-y}/\tilde\eps_n\right)$ holds for all $x,y\in\Omega$ and consequently} we obtain
\begin{align*}
    E_n(A_n)
    &\geq 
    \frac{1}{\eps_n}
    \int_\Omega 
    \left(\esssup_{x \in \Omega}\ksr{\phi}\left({\abs{x-y}}/{\tilde\eps_n}\right)\chi_{A_n}(T_n(x)) - \chi_{A_n}(T_n(y))\right)
    \rho_0(y)\de y
    \\
    &\qquad
    +
    \frac{1}{\eps_n}
    \int_\Omega 
    \ksr{
    \left(\esssup_{x \in \Omega}\ksr{\phi}\left({\abs{x-y}}/{\tilde\eps_n}\right)\chi_{A_n^c}(T_n(x)) - \chi_{A_n^c}(T_n(y))\right)
    }
    \rho_1(y)\de y
    \\
    &=
    \frac{\tilde\eps_n}{\eps_n}\Per_{\tilde\eps_n}(\hat{A}_n;\bm{\rho})
\end{align*}   
where $\hat{A}_n = \{\chi_{A_n} \circ T_n=1\}$.
This inequality together with \cref{thm:gammaCompact} establishes the compactness property in \cref{thm:gamma_graph}.
By assumption we have
\begin{align*}
    \limsup_{n\to\infty}
    \frac{\norm{T_n-\operatorname{id}}_{L^\infty(\Omega)}}{\eps_n}
    \leq 
    \limsup_{n\to\infty}
    \frac{\norm{T_n-\operatorname{id}}_{L^\infty(\Omega)}}{\delta_n}\frac{\delta_n}{\eps_n}
    =0\qquad\text{almost surely,}
\end{align*}
and hence $\frac{\Tilde{\eps}_n}{\eps_n}\to 1$ and $\Tilde{\eps}_n\to 0$ \ksr{by definition of $\Tilde{\eps}_n$ in \labelcref{eq:def_tilde_eps}}.
Furthermore, by assumption we have $\chi_{\hat A_n}=\chi_{A_n}\circ T_n\to\chi_A$ as $n\to\infty$.
Hence, we can take the limes inferior and use \cref{thm:liminf2} to get
\begin{align*}
    \liminf_{n\to\infty}
    E_n(A_n)
    \geq 
    \Per(A;\bm{\rho})\qquad\text{almost surely.}
\end{align*}
\end{proof}

\begin{lemma}[Discrete limsup inequality]\label{lem:graph_limsup}
Under the conditions of \cref{thm:gamma_graph} for any measurable $A\subset\R^d$ with probability one there exists a sequence of sets $(A_n)\subset X_n$ with $(\chi_{A_n},\nu_n)\to (\chi_A,\rho)$ in $TL^1(\Omega)$ and 
\begin{align*}
    \limsup_{n\to \infty} E_n(A_n) \leq \Per(A;\bm{\rho}).
\end{align*}
\end{lemma}
\begin{proof}
We can assume that $\Per(A;\bm{\rho})<\infty$.
Using \cref{cor:nearOptimal}, \ksr{for all $\eta>0$} there exists a set $A_\eta$ such that $\L^d(A\triangle A_\eta)\leq\eta$ and $\limsup_{n\to\infty}\Per_{\eps_n}(A_\eta;\bm{\rho})\leq\Per(A;\bm{\rho})+\eta$ for all sequences $\eps_n$ which converge to zero as $n\to\infty$.
We will abbreviate $\ksr{\Tilde{A}} := A_\eta$. Given the conditions of \cref{cor:nearOptimal}, we have that
\begin{equation} \label{eqn:minkCovnerges}
    \lim_{\epsilon \to 0}\mathcal{M}_\epsilon(\partial \Tilde{A} \cap \Omega) = \H^{d-1}(\partial \Tilde{A} \cap \Omega)
\end{equation}
by \cref{prop:minkowskiContent}.

The sequence $\ksr{A_n} := \tilde{A}\cap X_n$ then satisfies $(\chi_{A_n},\nu_n)\to (\chi_{\tilde{A}},\rho)$ in $TL^1(\Omega)$ and furthermore
\begin{align*}
    E_n(A_n)
    &=
    \frac{1}{\eps_n}
    \int_\Omega 
    \left(\max_{x \in X_n}\ksr{\phi}\left({\abs{x-y}}/{\eps_n}\right)\chi_{A_n}(x) - \chi_{A_n}(y)\right)
    \de\nu_n^0(y)
    \\
    &\qquad
    +
    \frac{1}{\eps_n}
    \int_\Omega 
    \ksr{
    \left(\max_{x \in X_n}\ksr{\phi}\left({\abs{x-y}}/{\eps_n}\right)\chi_{A_n^c}(x) - \chi_{A_n^c}(y)\right)
    }
    \de\nu_n^1(y)
    \\
    &=
    \frac{1}{\eps_n}
    \int_\Omega 
    \left(\esssup_{x \in \Omega}\ksr{\phi}\left({\abs{T_n(x)-T_n(y)}}/{\eps_n}\right)(\chi_{\Tilde{A}}\circ T_n)(x) - (\chi_{\Tilde{A}}\circ T_n)(y)\right)
    \rho_0(y) \de y
    \\
    &\qquad
    +
    \frac{1}{\eps_n}
    \int_\Omega 
    \ksr{
    \left(\esssup_{x \in \Omega}\ksr{\phi}\left({\abs{T_n(x)-T_n(y)}}/{\eps_n}\right)(\chi_{\Tilde{A}^c}\circ T_n)(x) - (\chi_{\Tilde{A}^c}\circ T_n)(y)\right)
    }
    \rho_1(y) \de y,
\end{align*}
\ksr{where we used a change of variables and utilized that by definition of $A_n$ it holds $\chi_{A_n}\circ T_n = \chi_{\Tilde{A}}\circ T_n$.}
Let us define a new scaling as $\Tilde{\eps}_n := \eps_n\left(1+ 2\norm{T_n-\operatorname{id}}_{L^\infty(\Omega)}/\eps_n\right)$.
\ksr{Analogously to the proof of \cref{lem:graph_liminf}, one obtains $\phi(\abs{T_n(x)-T_n(y)}/\eps_n)\leq \phi(\abs{x-y}/\Tilde{\eps}_n)$ for all $x,y\in\Omega$.}
This implies
\begin{align*}
    E_n(A_n)
    &\leq 
    \frac{1}{\eps_n}
    \int_\Omega
    \left(
    \esssup_{x \in \Omega}\phi(\abs{x-y}/\Tilde{\eps}_n)(\chi_{\Tilde{A}}\circ T_n)(x) - (\chi_{\Tilde{A}}\circ T_n)(y)\right)
    \rho_0(y) \de y
    \\
    &\qquad
    +
    \frac{1}{\eps_n}
    \int_\Omega
    \ksr{
    \left(
    \esssup_{x \in \Omega}\phi(\abs{x-y}/\Tilde{\eps}_n)(\chi_{\Tilde{A}^c}\circ T_n)(x) - (\chi_{\Tilde{A}^c}\circ T_n)(y)\right)
    }
    \rho_1(y) \de y
    \\
    &=
    \ksr{
    \frac{1}{\eps_n}
    \int_\Omega
    \left(
    \esssup_{x \in B(y,\Tilde{\eps}_n)\cap \Omega}(\chi_{\Tilde{A}}\circ T_n)(x) - (\chi_{\Tilde{A}}\circ T_n)(y)\right)
    \rho_0(y) \de y}
    \\
    &\qquad
    \ksr{
    +
    \frac{1}{\eps_n}
    \int_\Omega
    \left(
    \esssup_{x \in B(y,\Tilde{\eps}_n)\cap \Omega}(\chi_{\Tilde{A}^c}\circ T_n)(x) - (\chi_{\Tilde{A}^c}\circ T_n)(y)\right)
    \rho_1(y) \de y.}
\end{align*}    
\ksr{In the remainder of the proof we argue that we can replace every occurrency of $\chi_{\Tilde{A}}\circ T_n$ by $\chi_{\Tilde{A}}$ (and similar for $\Tilde{A}^c$) in the limit $n\to\infty$.
For this we first estimate the terms without essential suprema and then the terms with essential suprema.}

Defining the set $\hat{A}_n := \{\chi_{\Tilde{A}}\circ T_n=1\}$ we have $\chi_{\Tilde{A}}\circ T_n = \chi_{\hat{A}_n}$.
Furthermore, for every $x\in \Tilde{A}$ with $\dist(x,\Tilde{A}^c)>\norm{T_n-\operatorname{id}}_{L^\infty(\Omega)}$ we can find $\hat{x}\in X_n$ with $\abs{x-\hat{x}}\leq\norm{T_n-\operatorname{id}}_{L^\infty(\Omega)}$ and hence $\hat{x}\in \Tilde{A}$.
This implies $x\in \hat{A}_n$.
Similarly, one argues that $x\in \Tilde{A}^c$ with $\dist(x,\Tilde{A})>\norm{T_n-\operatorname{id}}_{L^\infty(\Omega)}$ implies $x\in \hat{A}_n^c$.
Hence, we obtain
\begin{align*}
    \int_\Omega\abs{\chi_{\Tilde{A}} \circ T_n(y) - \chi_{\Tilde{A}}(y)}\de y
    &=
    \int_\Omega\abs{\chi_{\hat{A}_n}(y) - \chi_{\Tilde{A}}(y)}\de y
    =
    \mathcal{L}^d\left(\hat{A}_n \triangle \Tilde{A}\right)
    \\
    &\leq 
    \mathcal{L}^d\left(\{x\in\Omega\st\dist(x,\partial \Tilde{A})\leq\norm{T_n-\operatorname{id}}_{L^\infty(\Omega)}\}\right)
    \\
    &\leq 
    (2+o(1))\H^{d-1}(\partial \Tilde{A} \cap \Omega)\norm{T_n-\operatorname{id}}_{L^\infty(\Omega)}
\end{align*}
where $o(1)\to 0$ for $n\to\infty$ comes from \labelcref{eqn:minkCovnerges}.

Using the notation $\Tilde{A}^{\oplus\eps} := \bigcup_{x\in\Tilde{A}}B(x,\eps)$, we get
\begin{align*}
    \int_\Omega
    \abs{
    \esssup_{x \in B(y,\tilde\eps_n)\ksr{\cap\Omega}}(\chi_{\Tilde{A}}\circ T_n)(x)
    -
    \esssup_{x \in B(y,\tilde\eps_n)\ksr{\cap\Omega}} \chi_{\Tilde{A}}(x)
    }
    \de y
    &\leq 
    \L^d\left(
    (\hat{A}_n)^{\oplus\tilde\eps_n} \triangle \Tilde{A}^{\oplus\tilde\eps_n}
    \right).
\end{align*}
We argue similarly to the previous case: 
Let us take $x \in (\Tilde{A}^{\oplus\tilde\eps_n})^c$ with $\dist(x,\Tilde{A}^{\oplus\tilde\eps_n})>2\norm{T_n-\operatorname{id}}_{L^\infty(\Omega)}$ and assume that $x\in(\hat{A}_n)^{\oplus\tilde\eps_n}$.
The former implies $\dist(x,\Tilde{A})>\tilde\eps_n + 2\norm{T_n-\operatorname{id}}_{L^\infty(\Omega)}$.
In contrast, the latter implies the existence of some $y\in\hat{A}_n$ with $\abs{x-y}<\tilde\eps_n$ and by definition of the set $\hat{A}_n$ we have $T_n(y)\in \tilde{A}$.
In particular, we obtain
\begin{align*}
    \abs{x-T_n(y)} \leq \abs{x-y} + \abs{y-T_n(y)} < \tilde\eps_n + \norm{T_n-\operatorname{id}}_{L^\infty(\Omega)}
\end{align*}
which is a contradiction.
We can argue symmetrically for points that lie within $\Tilde{A}^{\oplus\tilde\eps_n}$ with distance to the complement larger than $2\norm{T_n-\operatorname{id}}_{L^\infty(\Omega)}$. Using this we may estimate the symmetric difference of these sets by
\begin{align*}
    \L^d\left(
    (\hat{A}_n)^{\oplus\tilde\eps_n} \triangle \Tilde{A}^{\oplus\tilde\eps_n}
    \right)
    &\leq 
    \mathcal{L}^d\left(\{x\in\Omega\st\dist(x,\partial (\Tilde{A}^{\oplus\tilde\eps_n}))\leq 2\norm{T_n-\operatorname{id}}_{L^\infty(\Omega)}\}\right)
    \\
    &\leq 
    \mathcal{L}^d\left(\{x\in\Omega\st\dist(x,\partial \Tilde{A})\leq \tilde\eps_n + 2\norm{T_n-\operatorname{id}}_{L^\infty(\Omega)}\}\right)
    \\
    & \quad - \mathcal{L}^d\left(\{x\in\Omega\st\dist(x,\partial \Tilde{A})\leq \tilde\eps_n - 2\norm{T_n-\operatorname{id}}_{L^\infty(\Omega)}\}\right) 
    \\
    &\leq  (2+o(1))\H^{d-1}(\partial \Tilde{A} \cap \Omega) \left(\tilde \epsilon_n + 2\|T_n - \operatorname{id}\|_{L^\infty(\Omega)}\right)  \\
    & \quad - (2-o(1))\H^{d-1}(\partial \Tilde{A} \cap \Omega) \left(\tilde \epsilon_n - 2\|T_n - \operatorname{id}\|_{L^\infty(\Omega)}\right)
    \\
    &\leq 
    8(1+o(1))\H^{d-1}(\partial \Tilde{A})\norm{T_n-\operatorname{id}}_{L^\infty(\Omega)} + 4 o(1) \H^{d-1}(\partial \Tilde{A} \cap \Omega)  \tilde \epsilon_n,
\end{align*}
where $o(1)\to 0$ as $n\to\infty$ comes from the limit \labelcref{eqn:minkCovnerges}.

Combining the estimates we obtain
\begin{align*}
    &\phantom{{}={}}
    \frac{1}{\eps_n}
    \abs{
    \int_\Omega
    \left(
    \esssup_{x \in B(y,\tilde\eps_n)}(\chi_{\Tilde{A}}\circ T_n)(x) - (\chi_{\Tilde{A}}\circ T_n)(y)\right)
    \rho_0(y) \de y
    -
    \int_\Omega
    \left(
    \esssup_{x \in B(y,\tilde\eps_n)}\chi_{\Tilde{A}}(x) - \chi_{\Tilde{A}}(y)\right)
    \rho_0(y) \de y
    }
    \\
    &\leq 
    \frac{1}{\eps_n}
    \int_\Omega \left( 
    \abs{
    \esssup_{x \in B(y,\tilde\eps_n)}(\chi_{\Tilde{A}}\circ T_n)(x)
    -
    \esssup_{x \in B(y,\tilde\eps_n)} \chi_{\Tilde{A}}(x)
    }
    +
    \abs{
    (\chi_{\Tilde{A}}\circ T_n)(y)
    -
    \chi_{\Tilde{A}}(y)
    }
    \rho_0(y) \right)
    \de y
    \\
    &\leq 
    \norm{\rho_0}_{L^\infty(\Omega)}
    \H^{d-1}(\partial \Tilde{A} \cap \Omega)\left(
    \frac{8(1+o(1))\norm{T_n^0-\operatorname{id}}_{L^\infty(\Omega)}}{\eps_n} + 4o(1) \frac{\tilde \epsilon_n}{\epsilon_n} + (2+o(1))\frac{\norm{T_n^0-\operatorname{id}}_{L^\infty(\Omega)}}{\eps_n}\right)
\end{align*}
which converges to zero almost surely, as $n\to\infty$.
Similarly, we can argue for the term containing the complements $\Tilde{A}^c$ and in total we find, using $\Tilde{A}=A_\eta$, that
\begin{align*}
    \limsup_{n\to\infty}E_n(A_n)
    \leq 
    \limsup_{n\to\infty}
    \frac{\tilde\eps_n}{\eps_n}
    \Per_{\tilde\eps_n}(A_\eta;\bm{\rho})
    \leq 
    \Per(A;\bm{\rho}) + \eta.
\end{align*}
Finally, we perform a diagonalization argument over the sequence $A_{n,\eta} := A_\eta \cap X_n$ to obtain a recovery sequence $A_n$, converging to $A$ in $TL^1(\Omega)$ and satisfying
\begin{align*}
    \limsup_{n\to\infty}E_n(A_n)
    \leq 
    \Per(A;\bm{\rho}).
\end{align*}
\end{proof}


\section*{Acknowledgments}

Both authors acknowledge funding by the Deutsche Forschungsgemeinschaft (DFG, German Research Foundation) under Germany's Excellence Strategy - GZ 2047/1, Projekt-ID 390685813. 
Part of this work was done while LB was in residence at Institut Mittag-Leffler in Djursholm, Sweden during the semester on \textit{Geometric Aspects of Nonlinear Partial Differential Equations} in 2022, supported by the Swedish Research Council under grant no. 2016-06596.
Most of this work was done while LB was affiliated with the Hausdorff Center for Mathematics at the University of Bonn.

%
%
%
%

\printbibliography

\end{document}